\documentclass[a4paper,11pt]{article}
\usepackage[utf8]{inputenc}
\usepackage{amssymb}
\usepackage{amsthm}
\usepackage{amsmath}
\usepackage{anysize}
\usepackage{url}
\usepackage{color}
\usepackage{subfigure,graphicx}
\usepackage{overpic}
\usepackage{fullpage}
\usepackage{epstopdf}
\usepackage{paralist}
\usepackage{todonotes}
\usepackage{comment}

\newcommand{\Hil}{\mathcal{H}}

\newcommand{\Wcal}{\mathcal{W}}

\newcommand{\Z}{\mathbb{Z}}
\newcommand{\N}{\mathbb{N}}

\newcommand{\R}{\mathbb{R}}

\newcommand{\psihat}{\widehat{\psi}}
\newcommand{\psitilde}{\widetilde{\psi}}
\newcommand{\thetatilde}{\widetilde{\theta}}

\newcommand{\phihat}{\widehat{\phi}}

\newcommand{\Fcal}{\mathcal{F}}

\newcommand{\fa}{, \quad \text{ for all }}

\newcommand{\eps}{\varepsilon}
\newcommand{\Lp}[2]{L^{#1}{#2}}

\newcommand{\bfTPs}{\mathbf{T_{\Phi_s}}}
\newcommand{\bfTPw}{\mathbf{T_{\Phi_w}}}
\newcommand{\bfTP}{\mathbf{T_{\Phi}}}

\newcommand{\bfTPA}{\mathbf{T_{\Phi}^*}}
\newcommand{\bdWaveS}{\mathcal{W}} 

\def\diag{{\text{\rm diag}}}

\newtheorem{theorem}{Theorem}[section]

\newtheorem{assumptions}{Assumptions}[section]
\newtheorem{definition}[theorem]{Definition}
\newtheorem{proposition}[theorem]{Proposition}
\newtheorem{lemma}[theorem]{Lemma}

\DeclareMathOperator{\suppp}{supp \,}

\newcommand{\ExtOp}{\tilde{E}}
\newcommand{\ExtOpMod}{E}

\definecolor{ggreen}{rgb}{0,0.5,0.5}

\newcommand*{\changed}[1]{{\color{black}#1}}

\title{Anisotropic Multiscale Systems on Bounded Domains}
\author{Philipp Grohs \and Gitta Kutyniok \and Jackie Ma \and Philipp Petersen \and Mones Raslan}
\date{\today}

\begin{document}

\allowdisplaybreaks
\maketitle

\begin{abstract}
We provide a construction of multiscale systems on a bounded domain $\Omega \subset \R^2$ coined boundary
shearlet systems, which satisfy several properties advantageous for applications to imaging science and the numerical analysis
of partial differential equations. More precisely, we construct boundary shearlet systems that form frames for the Sobolev spaces $H^s(\Omega),s\in\N \cup \{0\},$ with controllable
frame bounds and admit optimally sparse approximations for functions, which are smooth apart from a curve-like discontinuity. We show that the constructed systems
allow to incorporate boundary conditions. Furthermore, for $s \geq 0$ and $f\in H^s(\Omega)$ we prove that weighted $\ell^2$ norms of the $L^2-$analysis coefficients of $f$ are equivalent to its $H^s(\Omega)$ norm. This yields in particular, that the reweighted systems are frames also for $H^{-s}(\Omega)$. Moreover, we demonstrate numerically, that the associated $L^2-$synthesis operator is also stable as a map to $H^s(\Omega)$ which, in combination with the previous result, strongly indicates that these systems constitute so-called Gelfand frames for $(H^s(\Omega), L^2(\Omega), H^{-s}(\Omega))$.
\end{abstract}

\section{Introduction}

In the past two decades, there has been extensive work related to the design of novel representation systems for functions with the intent to build systems which are optimally suited for sparse approximation of various signal
classes. The first breakthrough was the introduction of the multiscale system of wavelets \cite{TenLectures}
providing optimally sparse approximations of functions governed by point singularities while allowing a
unified concept of the continuum and digital realm to enable faithful implementations. Indeed, wavelets
have nowadays become a standard tool in, e.g., imaging science and the numerical analysis of partial
differential equations.

However, many \emph{multivariate functions} such as images or the solutions of various types of partial differential equations are governed by singularities along hypersurfaces. Additionally, it turned out that wavelets do not yield optimal approximation rates for these classes of functions \cite{DCartoonLikeImages2001}. Recently, significant progress has been achieved by the introduction of anisotropic multiscale systems such as ridgelets \cite{candes1998ridgelets}, curvelets \cite{CurveletsIntro}, and shearlets \cite{LLKW2007}, which are capable of optimally approximating certain classes of multivariate functions with singularities along hypersurfaces. One main drawback so far is the fact that all these systems are designed for $L^2(\R^2)$, whereas applications, in particular, adaptive
schemes for partial differential equations, typically require systems defined on a bounded domain $\Omega \subset \R^2$, where these systems need to constitute frames for Sobolev spaces $H^s(\Omega)$, $s\in \N\cup \{0\}$. 

In this paper, we will introduce the first class of anisotropic multiscale systems defined on very general domains $\Omega \subset \R^2$ which still exhibit optimally sparse approximation properties for functions with hypersurface singularities. Moreover, further
properties specifically necessary for their application in numerical approximation of PDEs such as the frame property for $H^s(\Omega)$, $s\in \N$ will be demonstrated. These systems are
constructed as hybrid systems combining wavelets and shearlets.

\subsection{Anisotropic multiscale systems}

Curvelets were the first generalization of wavelets that achieve optimal approximation rates for functions with curvilinear singularities. Specifically, they optimally approximate so-called cartoon-like functions \cite{DCartoonLikeImages2001}, which are compactly supported piecewise $C^2$-functions defined on $\mathbb{R}^2$ with a $C^2$ discontinuity curve. 

In fact, curvelets are capable of approximating such functions with a decay rate of the $L^2$-error of the best $N$-term
approximation by $N^{-1}$ up to logarithmic terms, whereas wavelets can only achieve a rate of $N^{-1/2}$.
This result from 2002 \cite{CurveletsIntro} can be considered a milestone in the area of applied harmonic analysis.

However, curvelets have the drawback of being based on rotations to provide directional sensitivity, which
prevents a unified treatment of the continuum and digital world. Hence, faithful implementations are not
available. Shearlets were introduced in 2006 to resolve this problem, leading to a class of systems which
also provide optimally sparse approximations of cartoon-like functions \cite{KLcmptShearSparse2011}---even for higher order derivatives \cite{Pet2015}---while allowing faithful implementations \cite{shearLab}.
Moreover, this concept allows a shearlet {\it frame} for $L^2(\R^2)$ with controllable frame bounds, consisting
of compactly supported elements \cite{KGLConstrCmptShear2012}. Let us stress that the notion of a frame allows for stable decomposition and reconstruction formulas,  see \cite{christensen2003introduction}.

\subsection{The problem of bounded domains}

Anisotropic multiscale systems are today extensively used in imaging science for tasks such as feature extraction
or inpainting, see, e.g., \cite{ESQD05,DK13,GK14}. Despite these successes, the use of these novel representation
systems for the numerical approximation of PDEs is still in its infancy---even though the solutions of a
large class of PDEs do admit hypersurface singularities. As examples, we mention elliptic PDEs with discontinuous
(or distributional) source terms or coefficients, boundary value problems on polygonal domains, or transport equations.

The main bottleneck in developing PDE solvers based on anisotropic multiscale systems is the fact that initially
these systems are constructed as representation systems, or frames, for functions defined on $\mathbb{R}^d$, while most PDEs are defined on a bounded domain $\Omega\subset \mathbb{R}^d$. Thus the
development of efficient PDE solvers crucially depends on the construction of anisotropic systems on bounded domains,
satisfying various boundary conditions. The attentive reader will have realized that also images are
defined on bounded domains. The fact that this issue did not cause a significant problem so far indicates that
the handling of the boundary for imaging tasks is not that sensitive. However, also in this range of applications, the
fact that the data lives on a bounded domain must be taken into consideration, especially in the theoretical (continuum) analysis. We expect that the usage of boundary-adapted systems may help to reduce boundary artifacts.

Even for wavelet systems, the adaptation to general bounded domains is a challenging problem which is by now, after
decades of research, quite well understood (see, for instance, \cite{harbrecht2006wavelets},\cite{CDDBoundaryWavelets}). Nonetheless, several open questions remain. The construction of anisotropic multiscale systems on bounded domains
is even much more challenging. While the MRA structure of wavelet systems \cite{TenLectures} yields a powerful tool
for the construction of boundary wavelet frames, no such structure seems to be available for more general multiscale
systems. 
In fact, for anisotropically scaled and directed systems one can imagine that the anisotropically shaped support can intersect the boundary to various degrees and at various angles requiring a careful adjustment of each single element. 

\subsection{Adaptive schemes and frames}\label{sec:adaptframe}

Several important steps have already been made in the direction of utilizing representation systems from applied harmonic
analysis for adaptive solvers for different types of partial differential equations. A breakthrough was achieved by Cohen, Dahmen, and DeVore in the seminal paper \cite{cohen2001adaptive} by introducing a provably optimal adaptive wavelet-based solver for elliptic PDEs with solutions exhibiting only point singularities. Several further steps could be achieved in the realm of elliptic PDEs, for example, extending this concept to more flexible frames instead of orthonormal bases, see, e.g., \cite{stevenson2003adaptive,dahlke2007adaptive, DahRWFS2007AdaStepDec}.

The exploitation of anisotropic frame systems for the numerical solution of PDEs is a rather new topic of research, 
presumably due to the previously described difficulty of adapting those systems to bounded domains.
We recall the key properties an anisotropic frame system needs to satisfy to be suitable for applications in the discretization of PDEs.

\begin{enumerate}[(P1)]
\item {\it Boundary conditions.} The system should allow to impose boundary conditions.
\item {\it Quadrature.} The elements of the system should be spatially localized and ideally \changed{piecewise polynomial with low polynomial order} to allow for efficient quadrature.
\item {\it Frame properties.} The constructed system should form a frame for $H^s(\Omega)$, where $\Omega$ is the
bounded domain and $s\in \N \cup \{0\}$. 
\item {\it Characterization of Sobolev spaces.} The analysis and synthesis operator of the system should yield stable maps between 
weighted $\ell^2$ spaces and Sobolev spaces. In particular, such a property is achieved when the system forms a 
frame for $H^s(\Omega)$ or it forms a so-called Gelfand frame for Sobolev spaces. The discretization of an 
elliptic PDE using a Gelfand frame for a Sobolev space yields, after a simple diagonal preconditioning, a uniformly 
well-conditioned linear system which can then be solved numerically by iterative methods such as damped Richardson 
iteration or conjugate gradients \cite{dahlke2007adaptive}.  
\item {\it Optimal sparse approximation.} It is crucial that the system provides optimally sparse approximation 
properties of the solutions to ensure that those can be approximated at an optimal asymptotic ratio between 
computational work and accuracy. Ideally, this should include even functions with discontinuities ending at the 
boundary. 
\item {\it Compressibility.}  For this approximation rate to be realized by a numerical solver, it is 
furthermore needed that the resulting discretization matrix is compressible in the sense of \cite{cohen2001adaptive}.
\end{enumerate}
Notable first steps towards anisotropic frame systems for the numerical solution of PDEs have already been
taken in \cite{EGO14_1305,GO14_1334}, where optimal adaptive ridgelet-based solvers are constructed for linear advection
equations and \cite{dahmen2014efficient,dahmen2012adaptive,DKLSW15}, where a shearlet-based construction is used to 
solve general advection equations. Also related is the work \cite{demanet2009wave,candes2004curvelet}, where frames 
of wave atoms, respectively curvelets, are used for the efficient representation and computation of wave propagators. 
Despite these successes, none of the aforementioned works successfully addressed the essential issue of problem 
formulations on finite domains with non-periodic boundary conditions.

\subsection{Our contribution}

In this paper, we present a construction of an anisotropic multiscale system on bounded domains $\Omega \subset \mathbb{R}^2$, 
which satisfies most of the aforementioned desired properties ((P1)--(P6)). We regard this as a significant step
towards utilizing anisotropic frame systems to derive provably optimal solvers for the numerical solution of appropriate types of PDEs. The novel system coined \emph{boundary shearlet system} is a hybrid system, consisting of shearlet elements to provide the
optimal approximation rate for anisotropic structures and wavelet elements for handling the boundary. More precisely,
we start with a compactly supported shearlet frame for $L^2(\mathbb{R}^2)$ as constructed, for instance, in
\cite{KGLConstrCmptShear2012}. From this frame, we only choose the elements with support fully contained
in $\Omega$. Since this is by far not a complete system for $L^2(\mathbb{R}^2)$ and certainly cannot handle
boundary conditions, we augment it by a carefully chosen selection of boundary-adapted wavelets as constructed, for instance, in \cite{CohDauVial}. We succeed in establishing the following properties from ((P1)--(P6)):

\begin{enumerate}[(P1)]
\item {\it Boundary conditions.} The usage of boundary wavelets will ensure that any boundary condition can be imposed.
\item {\it Quadrature.} The elements of such a boundary shearlet system are by construction spatially localized and can be chosen to be piece-wise polynomial.
\item {\it Frame properties.} The constructed boundary shearlet system forms a frame for $L^2(\Omega)$, where 
$\Omega$ is the bounded domain, due to the careful selection of elements from a shearlet system and a boundary
wavelet system. Additionally, we will demonstrate that for $s \in \N$ one obtains a frame for $H^s(\Omega)$ after proper reweighting of the frame's elements. This will be established in Theorem \ref{thm:FrameProperty}.
\item {\it Characterization of Sobolev spaces.}  Theorem \ref{thm:AnalysisCharOfHOmega} shows that our boundary shearlet system is
capable of characterizing the norm of Sobolev spaces, i.e., if $(\varphi_n)_{n\in \N}$ denotes the shearlet frame 
on $\Omega$ then for all $f\in H^s(\Omega)$ we have that
$
\|f\|_{H^s{(\Omega)}} \sim \| (w_n \langle f, \varphi_n\rangle)_{n\in \N} \|_2,
$
for some sequence of weights $(w_n)_{n\in \N}$ depending only on $\changed{s \geq 0}$. This implies one condition necessary for deriving a Gelfand frame for the Gelfand triple $(H^{s}(\Omega),L^2(\Omega),H^{-s}(\Omega))$. The 
other condition will be shown numerically since it requires the dual frame, which is presently not available in 
analytic form. This gives strong evidence that the boundary shearlet system indeed constitutes a Gelfand frame 
which is sufficient for applications in the numerical solution of elliptic PDEs \cite{dahlke2007adaptive}.
The characterization of 
Sobolev spaces also implies that our system constitutes a frame for the Hilbert space $H^{-s}(\Omega)$ which might
be relevant for the potential numerical solution of boundary integral equations \cite{sauter2010boundary}. 
\item {\it Optimal sparse approximation.} Theorem \ref{thm:OptSpaApproxOnBdDom} proves that a boundary shearlet system
provides optimally sparse approximation of an even extended class of cartoon-like functions, which also include 
singularity curves traversing the boundary. 
\end{enumerate}

Since the problem of compressibility is associated with a specific PDE problem and is known to be very involved, we believe it goes far beyond the scope of this paper. Property (P6) will be a subject of future work. In summary, except for compressibility of the discretization matrix, our contribution provides a fully satisfactory answer to the desired properties ((P1)--(P5)).

\subsection{Expected impact}
We anticipate our results to have the following impacts:
\begin{itemize}
\item {\em Numerical solution of partial differential equations.}
Our work represents a step in a larger research program, which is to construct and apply specifically designed frames for the numerical solution of partial differential equations. 
The construction of this paper is expected to lead to the first adaptive algorithm for elliptic PDEs with solutions possessing singularities along hypersurfaces which can be executed in optimal complexity. Nonetheless, we wish to emphasize that,
while the results concerning the construction of directional frames presented in this paper represent a significant advance there still exist open problems already in this realm that need to be addressed in the future. 

First, the compressibility in the sense of \cite{cohen2001adaptive} of the matrix representation of elliptic PDEs in our
constructed system needs to be studied. Second, our results show that Sobolev norms can be characterized by weighted
$\ell^2$ norms on the transform coefficients. The construction of optimal adaptive PDE solvers in the spirit of \cite{cohen2001adaptive,dahlke2007adaptive} requires slightly more, namely that the representation system constitutes a Gelfand frame
\cite{dahlke2007adaptive}. 
In Section \ref{sec:numerics} we are so far only able to verify these properties numerically. 
Third, also on the practical side several issues remain such as the development of efficient quadrature rules for the representation of elliptic operators, say, in our representation system.
\item {\em Imaging sciences.} Images are naturally supported on a rectangular domain. A common approach so far was to theoretically
analyze algorithmic procedures from applied harmonic analysis for, for instance, denoising, inpainting, or feature extraction, see \cite{ESQD05,DK13,GK14}, in $L^2(\R^2)$ while disregarding any boundary issues. Consequently, also their digitization
was not particularly designed to handle a bounded domain. With the construction of multiscale systems on bounded domains, we now open the door to a unified concept of the continuum and digital realm for data on bounded domains. 
\item {\em Hybrid systems and sparse approximation.}
While there already exists work on hybrid constructions using systems
from applied harmonic analysis \cite{DK13, EasLN2013Denoising},
approximation properties of hybrid systems have not been studied
before. Thus, one can see our construction as one first step in a
line of research introducing more flexible systems by introducing hybridization of well-studied systems in a way such that special properties---in our case sparse approximation properties combined with boundary
adaption---are maintained.
\end{itemize}

\subsection{Outline}

The paper is organized as follows. Section \ref{sec:review} is devoted to discussing the main definitions and results for
boundary-adapted wavelet systems and shearlet systems. In the same section, we also state typical assumptions imposed on boundary wavelet and shearlet systems, which will play a crucial role later. The precise definition
of multiscale anisotropic directional systems on bounded domains, coined boundary shearlet systems, and the analysis
of their frame properties for $L^2(\Omega)$ and $H^s(\Omega)$, $s\in \N$, can be found in Section \ref{subsec:NewConstruction}. In Section \ref{sec:caracH}, we show how
to construct shearlet systems that characterize $H^s(\R^2)$ by their analysis coefficients and how this gives rise to
shearlet systems on bounded domains that characterize $H^s(\Omega)$. We further show partial results of the Gelfand frame property in the same section. In Section
\ref{sec:ApproxPropOmega}, it is demonstrated that the newly constructed systems provide optimal $L^2-$approximation rates for a more
general class of cartoon-like functions than in previous literature considered. Finally, in Section \ref{sec:numerics} we complement the theoretical analysis of the Gelfand frame property of Section \ref{sec:caracH} by a numerical analysis of the Gelfand frame property and verify the stability and compressibility properties of our system numerically.


\section{Review of wavelet and shearlet systems} \label{sec:review}

Since wavelets on bounded domains and shearlets will be the key ingredient in the following construction of boundary shearlets systems, this section should serve as a review of their main definitions and properties. Before that, we briefly recall frames and Riesz bases for Hilbert spaces, which are generalizations of orthonormal bases.

\subsection{Frames and Riesz bases}\label{sec:FramesRiesz}

Many representation systems for $L^2(\R^2)$ which occur within the framework of harmonic analysis do not yield orthonormal bases, since they are redundant. Nonetheless, they still exhibit stability
properties in the sense of constituting a frame for $L^2(\R^2)$ with controllable frame bounds. Recall that a family of elements
$(\varphi_n)_{n\in \N}$ in a Hilbert space
$\mathcal{H}$ forms a {\em frame} for $\mathcal{H}$, if there exist $0<A\leq B<\infty$ such that
$$
A\|f\|_{\mathcal{H}}^2 \leq \sum_{n\in \N} | \langle \varphi_n, f\rangle_{\mathcal{H}} |^2 \leq B\|f\|^2_{\mathcal{H}} \fa f\in \mathcal{H}.
$$
If only the second inequality holds, then the system $(\varphi_n)_{n\in \N}$ is called a \emph{Bessel sequence}. Associated to every
Bessel sequence $(\varphi_n)_{n\in \N}$ is the \emph{analysis operator} $T$ given by
\[
T : \mathcal{H} \to \ell^2(\N), \quad f \mapsto (\langle f, \varphi_n\rangle_{\mathcal{H}})_{n\in \N}.
\]
The inner products $(\langle f, \varphi_n\rangle_\mathcal{H})_{n \in \N}$ are termed {\it analysis coefficients}, in contrast to the coefficients of an expansion of $f$ in the system $(\varphi_n)_{n \in \N}$ which are referred to as {\it synthesis coefficients}.
If $(\varphi_n)_{n\in \N}$ constitutes a frame for $\mathcal{H}$, it can be shown that the {\em frame operator} $S := T^*T$ is
a boundedly invertible operator \cite{christensen2003introduction}. Aiming to reconstruct any $f$ from $Tf$, one
defines another frame $(\varphi_n^d)_{n\in \N}$, the so-called \emph{canonical dual frame} of $(\varphi_n)_{n\in \N}$, by
\[
\varphi_n^d: = S^{-1}\varphi_n.
\]
This leads to the formulae
\begin{align*}
f = \sum_{n\in \N} \langle f, \varphi_n\rangle_\mathcal{H} \varphi_n^d = \sum_{n\in \N} \langle f, \varphi_n^d\rangle_\mathcal{H} \varphi_n \fa f\in \mathcal{H},
\end{align*}
of which the first part is the desired reconstruction formula, and the second can be regarded as an expansion of $f$ in terms of the frame
$(\varphi_n)_{n\in \N}$. If $(\varphi_n)_{n\in\N}$ fulfills the additional property that for all finitely supported, scalar-valued sequences $(c_n)_{n\in\N}$ there holds 
\[\|(c_n)_{n\in\N}\|_{\ell^2}^2\sim \|\sum_{n\in\N}c_n \varphi_n\|_{\Hil}^2,\] it is called \emph{Riesz basis} for $\Hil$ and its canonical dual frame $(\varphi_n^d)_{n\in\N}$ is also a Riesz basis. Furthermore, for all $n,n'\in \N$ there holds 
\[\langle \varphi_n,\varphi_{n'}^d \rangle_{\Hil}=\delta_{n,n'},\] i.e., $(\varphi_n)_{n\in\N},(\varphi^d_{n})_{n\in\N}$ is a \emph{pair of biorthogonal Riesz bases} for $\Hil.$

\subsection{Boundary-adapted wavelet systems}\label{sec:bdWavelets}


In the sequel we will work with wavelet systems on a bounded domain $\Omega \subset \R^2$. Later on we will assume a number of properties that the wavelet systems should have. For now we only stipulate that they should be indexed in a specific way. 
\begin{definition} \label{def:bws}
A set $\bdWaveS \subset L^2(\Omega)$ is called a \emph{boundary wavelet system,} if for some $J_0 \in \Z$ it can be written as
\begin{align*}
\bdWaveS = \{ \omega_{J_0, m, 0}: m \in K_{J_0}\} \cup \{ \omega_{j, m, \upsilon}: j \geq J_0, m \in K_{j}, \upsilon = {1,2,3}\},
\end{align*}
where $K_j\subset \Omega$, $\# K_j \sim 2^{2 j}$. We denote the index set by
$$
\Theta: =\{(J_0 ,m, 0): m\in K_j\} \cup \{(j,m, \upsilon): j\geq 0, m\in K_j ,\upsilon = {1,2,3}\}.
$$
Furthermore, we call $\Wcal,\Wcal^d$ a \emph{pair of biorthogonal wavelet Riesz bases} for $L^2(\Omega),$ if $\Wcal, \Wcal^d$ are boundary wavelet systems and $\Wcal,\Wcal^d$ is a pair of biorthogonal Riesz bases for $L^2(\Omega).$
\end{definition}
Explicit constructions of boundary wavelet systems that yield pairs of biorthogonal wavelet Riesz bases for $L^2(\Omega)$ can be found e.g. in \cite{CohDauVial, DahS1998,DahKU1996stokes,CanTU1999WaveElMeth,bitt2006biortho,primbs2006stabile}. If one wants to work with orthonormal wavelet bases, one could use orthonormal spline wavelets as in \cite{splinewave96,splinewave99,splinewave10}. Furthermore, conditions under which reweighted boundary wavelet systems form frames for $H^s(\Omega)$ have been examined for instance in \cite{DahWaveMultOP}. Additionally, conditions were established in \cite{DahWaveMultOP} such that wavelet systems characterize Sobolev spaces in the sense that 
\begin{align*}
C_{\mathrm{w}}\|f\|_{H^{s}(\Omega)}^2  \leq \sum_{(j,m,\upsilon)\in  \Theta} 2^{2js}|\left \langle f, \omega_{j,m,\upsilon}\right \rangle_{L^2(\Omega)} |^2 \leq D_{\mathrm{w}}\|f\|_{H^{s}(\Omega)}^2,
\end{align*}
for some $C_{\mathrm{w}},D_{\mathrm{w}}>0$ and for all $ f\in H^{s}(\Omega)$. This property will turn out to be crucial for the proof of the characterization of Sobolev spaces by boundary shearlet systems.

\subsubsection{Assumptions on boundary wavelets}
Now we introduce some admissibility conditions imposed on boundary wavelet systems which will be central to our later construction of boundary shearlets. 
More precisely, we are interested in the following list of desiderata, a boundary wavelet system $\mathcal{W}$ as defined in
Definition \ref{def:bws} can satisfy on $\Omega$:

\begin{assumptions}\label{ass:Wave}
Let $s\geq 0$ and $\mathcal{W}$ be a boundary wavelet system.  
\begin{itemize}

\item[(W1)] There exists a boundary wavelet system system $\Wcal^d$ such that $\Wcal,\Wcal^d$ is a pair of biorthogonal Riesz bases for $L^2(\Omega).$ 

\item[(W2)] The wavelet system $\mathcal{W}$ characterizes $H^s(\Omega)$, i.e., there exists $0<C_{\mathrm{w}}\leq D_{\mathrm{w}}<\infty$ such that, for all $f\in H^{s}(\Omega)$,
\[
C_{\mathrm{w}}\|f\|_{H^{s}(\Omega)}^2  \leq \sum_{(j,m,\upsilon) \in \Theta} 2^{2js}|\left \langle f, \omega_{j,m,\upsilon}\right \rangle_{L^2(\Omega)} |^2 \leq D_{\mathrm{w}}\|f\|_{H^{s}(\Omega)}^2.
\]

\item[(W3)] $(2^{-js}\omega_{j,m,\upsilon})_{(j,m,\upsilon) \in \Theta}$ is a frame for $H^s(\Omega)$ and there exists a dual frame \\
$\mathcal{W}^d_{H^s}=(2^{-js}\omega_{j,m,\upsilon})^d_{(j,m,\upsilon) \in \Theta}$ and for all $(j,m,\upsilon)\in \Theta$ with ${\partial \Omega \cap \suppp (2^{-js}\omega_{j,m, \upsilon})^d = \emptyset}$ it holds that 
\begin{align*}
|\widehat{(2^{-js}\omega_{j,m, \upsilon}})^d(\xi)| \lesssim 2^{-js}\cdot 2^{-j}\frac{\min\{1,|2^{-j}\xi_i|^{\alpha_{\mathrm{w}}}\}}{\max\{1,|2^{-j}\xi_1|^{\beta_{\mathrm{w}}}\}\max\{1,|2^{-j}\xi_2|^{\beta_{\mathrm{w}}}\}}, 
\end{align*}

for at least one $i\in \{1,2\}$ some $\alpha_{\mathrm{w}},\beta_{\mathrm{w}}>0$ and all $\xi \in \R^2$.
Here the Fourier transform is to be understood on $\Lp{2}{(\R^2)}$ after extension by $0$ on $\R^2$. Furthermore, assume that the elements of $\bdWaveS^d_{H^s}$ have compact support and let 
$$
q_{\mathrm{w}}^{(0)} := \inf \{ q>0 : \suppp (2^{-js}\omega_{j,m, \upsilon})^d\subset B_{2^{-j}q}(m) \fa (j,m,\upsilon) \in \Theta\}>0.$$
\item[(W4)] The elements of $\bdWaveS$ have compact support and 
$$
q_{\mathrm{w}}^{(1)} := \inf \{q>0: \suppp \omega_{j,m, \upsilon}\subset B_{2^{-j}q}(m) \fa (j,m,\upsilon) \in \Theta\} >0.
$$
Moreover, we have that 
$$
|m - m'| \geq 2^{-j}q_{\mathrm{w}}^{(2)}\fa j\geq J_0 \text{ and }m,m' \in K_j,m \neq m',
$$
for some $q_{\mathrm{w}}^{(2)}>0.$
\end{itemize}
\end{assumptions}

Note that the decay condition of (W3) is always satisfied as soon as the generating wavelet is sufficiently smooth and possesses sufficiently many vanishing moments.


\subsection{Shearlets}\label{sec:shearlets}

Shearlet systems are designed to provide optimally sparse approximations of a model class of functions which are governed by
curvilinear singularities while allowing a faithful implementation with their construction being based on the framework of
affine systems. They were first introduced by Guo, Labate, Lim, Weiss, and one of the authors in \cite{GKL2006,LLKW2007}. Since
we aim to use shearlets as interior elements for the construction of our anisotropic multiscale systems on bounded domains we briefly recall the definition and properties of (cone-adapted) shearlet systems below. 

\subsubsection{Construction of (cone-adapted) shearlet systems}

The construction of shearlet systems is based on \emph{anisotropic scaling} and \emph{shearing}. To state the precise definition,
for $j \in \N, k \in \Z$, we denote the \emph{anisotropic scaling matrices} $A_j$ and the \emph{shear matrices} $S_k$ by
\begin{align*}
    A_j := \diag(2^j,2^{\frac{j}{2}}) =  \begin{pmatrix} 2^j & 0 \\ 0 & 2^{\frac{j}{2}} \end{pmatrix}, \quad S_k := \begin{pmatrix} 1 & k \\ 0 & 1 \end{pmatrix}.
\end{align*}
Then a \emph{(cone-adapted) shearlet system} is defined as follows:

\begin{definition}[\cite{KGLConstrCmptShear2012}]\label{def:ShearletSystem}
Let $\phi, \psi, \widetilde{\psi} \in L^2(\R^2)$, $c= (c_1, c_2) \in \R^2$ with $c_1,c_2>0$. Then the \emph{(cone-adapted) shearlet system} is defined by
\[
 \mathcal{SH}(\phi, \psi, \widetilde{\psi}, c) = \Phi(\phi, c_1) \cup \Psi(\psi, c) \cup \tilde{\Psi}(\widetilde{\psi}, c),
\]
where
\begin{align*}
\Phi(\phi, c_1) &:= \left \{ \phi(\cdot - c_1 m) :m\in \Z^2 \right\},\\
\Psi(\psi, c) &:= \left\{ \psi_{j,k,m} = 2^{\frac{3j}{4}}\psi(S_k A_{j}\cdot -  M_c m ): j\in \N, |k| \leq \left\lceil2^{\frac{j}{2}}\right\rceil, m\in \Z^2 \right\},\\
\tilde{\Psi}(\widetilde{\psi}, c) &:= \left\{ \widetilde{\psi}_{j,k,m} = 2^{\frac{3j}{4}}\widetilde{\psi}(S_k^T \tilde{A}_{j}\cdot -  M_{\tilde{c}} m ): j\in \N, |k| \leq
\left\lceil2^{\frac{j}{2}}\right\rceil, m\in \Z^2 \right\},
\end{align*}
with $M_c:= \diag( c_1, c_2 )$, $M_{\tilde{c}} := \diag( c_2 , c_1 )$, and $\tilde{A}_{2^j} := \diag(2^{\frac{j}{2}},2^{j})$.
\end{definition}
We shorten the notation by defining
\begin{align*}
\psi_{j,k,m,\iota} := \left \{\begin{array}{l l}
\psi_{j,k,m}&  \text{ if }\iota = 1,\\
\phi(\cdot - c_1 m)  & \text{ if }\iota = 0, \\
\psitilde_{j,k,m}&  \text{ if }\iota = -1,\\
\end{array} \right.
\end{align*}
and denote the index set for the full shearlet system by
$$
\Lambda:= \left\{(j,k,m,\iota) \, : \, \iota \in \{-1,0,1\}: \ |\iota|j = j \geq 0,\ |k| \leq |\iota| \left\lceil2^{\frac{j}{2}}\right\rceil,\ m \in \Z^2\right\}.
$$

\subsubsection{Frames of shearlet systems}

By imposing some weak conditions on the generators $\phi$, $\psi$, and $\widetilde{\psi}$, the system $\mathcal{SH}(\phi, \psi, \widetilde{\psi}, c)$
forms a frame for $L^2(\R^2)$. In fact, the following result is proved in \cite{KGLConstrCmptShear2012}.

\begin{theorem}[\cite{KGLConstrCmptShear2012}]\label{theorem:shearletsframe}
Let $\phi, \psi \in L^2(\R^2)$ such that
\begin{align*}
    |\phihat(\xi_1, \xi_2)| \leq C_1 \min \{1, | \xi_1|^{-\beta} \} \min \{ 1, |\xi_2|^{-\beta} \}
\end{align*}
and
\begin{align*}
    |\psihat(\xi_1, \xi_2)| \leq C_2 \min \{1, | \xi_1|^{\alpha} \} \min \{1, | \xi_1|^{-\beta} \} \min \{ 1, |\xi_2|^{-\beta} \},
\end{align*}
for some constants $C_1,C_2 >0$, and $\alpha > \beta >3$ and almost every $(\xi_1, \xi_2) \in \R^2$. Then, there exists $B>0$ such that for any $c = (c_1, c_2) \in \R^{+}\times \R^{+}$ the cone-adapted shearlet system
$\mathcal{SH}(\phi,\psi, \psitilde, c)$ forms a Bessel sequence for $L^2(\R^2)$ with Bessel bound $B/|\det(M_c)|$.
Further, let $\psitilde (x_1,x_2) = \psi(x_2, x_1)$
and assume there exists a positive constant $A>0$ such that
\begin{align*}
    |\phihat(\xi)|^2 + \sum_{j \in \N} \sum_{|k| \leq \lceil 2^{j/2}  \rceil} |\psihat(S_k^T (A_j)^{-1} \xi)|^2 + \sum_{j \in \N} \sum_{|k| \leq \lceil 2^{j/2}  \rceil} |\widehat{\psitilde}(S_k (\widetilde{A}_j)^{-1} \xi)|^2 > A
\end{align*}
holds almost everywhere. Then there exists $c = (c_1, c_2) \in \R^{+}\times \R^{+}$ such that the cone-adapted shearlet system
$\mathcal{SH}(\phi,\psi, \psitilde, c)$ forms a frame for $L^2(\R^2)$.
\end{theorem}

Given the results for boundary wavelets in the preceding Subsection \ref{sec:bdWavelets} one can pose the question, whether shearlets can also characterize Sobolev spaces as well as form frames for $H^s(\R^2)$ after reweighting. While there are some embedding results of Besov spaces
into shearlet spaces and vice versa, see \cite{LMN2013}, \cite{DahSteTes}, we are not aware of both a precise characterization of Sobolev
spaces $H^s(\R^2)$ or the construction of shearlet frames for $H^s(\R^2)$ by cone-adapted compactly supported shearlets. In Subsection \ref{subsec:FrameProp} and \ref{subsec:Hsomega} we outline how such constructions can be achieved.
\subsubsection{Assumptions on shearlet systems}

As for wavelets, we can introduce certain assumptions that shearlets need to fulfill throughout this paper.

\begin{assumptions} \label{ass:Shear}
Let $\mathcal{SH}(\phi,\psi,\psitilde,c)$ be a shearlet system as in Definition \ref{def:ShearletSystem} and $s\geq 0.$

\begin{itemize}
\item[(S1)] $\mathcal{SH}(\phi, \psi, \widetilde{\psi}, c)$  constitutes a frame for $L^2(\R^2)$.

\item[(S2)] For some $C_1,C_2 >0$ the decay conditions
\[
|\widehat{\psi}(\xi_1,\xi_2)| \leq C_1 \frac{\min\{1,|\xi_1|^{\alpha_{\mathrm{sh}}}\}}{\max\{1,|\xi_1|^{\beta_{\mathrm{sh}}}\}\max\{1,|\xi_2|^{\beta_{\mathrm{sh}}}\}}
\]
and
$$
|\widehat{\psitilde}(\xi_1,\xi_2)| \leq C_2 \frac{\min\{1,|\xi_2|^{\alpha_{\mathrm{sh}}}\}}{\max\{1,|\xi_1|^{\beta_{\mathrm{sh}}}\}\max\{1,|\xi_2|^{\beta_{\mathrm{sh}}}\}}
$$
are obeyed for all $(\xi_1, \xi_2) \in \R^2$ and some $\alpha_{\mathrm{sh}},\beta_{\mathrm{sh}}>0$.

\item[(S3)]
The shearlet system $\mathcal{SH}(\phi, \psi, \widetilde{\psi}, c)$ characterizes Sobolev spaces of order $s$, i.e., there
exist $0 < A_{\mathrm{sh}}\leq B_{\mathrm{sh}} < \infty$ such that for all $f\in H^{s}(\R^2)$
\[
\hspace*{-0.5cm}
A_{\mathrm{sh}}\|f\|_{H^{s}(\R^2)}^2  \leq \sum_{(j,k,m,\iota) \in \Lambda} 2^{2js}|\left \langle f, \psi_{j,k,m,\iota}\right \rangle_{L^2(\R^2)} |^2
\leq B_{\mathrm{sh}}\|f\|_{H^{s}(\R^2)}^2.
\]
\item[(S4)] $(2^{-js}\psi_{j,k,m,\iota})_{(j,k,m,\iota) \in \Lambda}$ is a frame for $H^s(\R^2)$ with dual frame $(2^{-js}\psi_{j,k,m,\iota})^d_{(j,k,m,\iota) \in \Lambda}$.

\item[(S5)] For all $(j,k,0,\iota)\in \Lambda$ and some $q_{\mathrm{sh}}>0$ we have that
$$
\suppp \psi_{j,k,0,\iota} \subset B_{2^{-j/2}q_{\mathrm{sh}}/2}(0).
$$
\end{itemize}
\end{assumptions}


In applications, we are interested in shearlet systems which fulfill all of the requirements (S1)-(S5). However, not all of them will be needed in the proofs of the results to follow which is why we will always only assume those conditions that are necessary for the formulation of the underlying statements.


\subsubsection{Approximation properties} \label{sec:ApproxPropPlane}

The target property shearlet systems were designed to satisfy is optimal approximation of cartoon-like functions.
This class of functions was first introduced by Donoho in \cite{DCartoonLikeImages2001} as a suitable model for natural images,
and consists---roughly speaking---of compactly supported functions which are $C^2$ apart from a $C^2$-discontinuity curve.

One key ingredient in the definition of cartoon-like functions is the class of sets $STAR^2(\nu)$, which are star-shaped sets
with $C^2$ boundary and curvature bounded by $\nu>0$. Then the definition of cartoon-like functions reads as follows.

\begin{definition}\label{def:CartoonLikeRTwo}
For $\nu >0$, let $\mathcal{E}^2(\nu)$ denote the set of functions $f \in L^2(\R^2)$, for which there exist some $D\in STAR^2(\nu)$
and $f_i \in C^2(\R^2)$ with compact support in $[0,1]^2$ as well as $\|f_i\|_{C^2} \leq 1$ for $i = 1,2$, such that
$$
f = f_1 + \chi_D f_2.
$$
The elements $f\in \mathcal{E}^2(\nu)$ are called \emph{cartoon-like functions}.
\end{definition}

In the same paper \cite{DCartoonLikeImages2001}, Donoho presented the first optimality result concerning the approximation
rate for this class of functions by more or less arbitrary representation systems. In fact, in \cite{DCartoonLikeImages2001}
it was shown that for any representation system  {$(\varphi_n)_{n\in \N} \subset L^2(\R^2)$}, the minimally achievable uniform asymptotic
approximation error for $ \mathcal{E}^2(\nu)$ is
$$
\sup_{f\in \mathcal{E}^2(\nu)}\|f-f_N\|_{L^2(\R^2)}^2 = O(N^{-2}) \quad \mbox{as } N \to \infty,
$$
where $f_N$ denotes the non-linear best $N$-term approximation of $f$ obtained by choosing the $N$ largest coefficients
through \emph{polynomial depth search}. The term polynomial depth search means that the $i$-th term in the expansion can
only be chosen in accordance with a selection rule $\sigma(i,f)$, which obeys $\sigma(i,f) \leq \pi(i)$ for a fixed
polynomial $\pi(i)$, see \cite{DCartoonLikeImages2001}; thereby avoiding artificial representation systems which are dense in $L^2(\R^2)$.

In \cite{KLcmptShearSparse2011}, one of the authors together with Lim proved that there exist shearlet systems such that any dual frame can achieve
this optimal rate up to a $\log$ factor. More precisely, the following theorem was shown.

\begin{theorem}[\cite{KLcmptShearSparse2011}]\label{thm:compShearAreOptSparse}
Let $c>0$, and let $\phi, \psi, \psitilde \in L^2(\R^2)$ be compactly supported. Suppose that, in addition, for all
$\xi = (\xi_1,\xi_2) \in \R^2$ the shearlet $\psi$ satisfies
\begin{compactenum}[(i)]
\item $|\widehat{\psi}(\xi)|^2 \leq C_1\cdot \min(1, |\xi_1|^\alpha)\cdot \min(1, |\xi_1|^{-\beta}) \cdot \min(1, |\xi_2|^{-\beta})$ and
\item $ |\frac{\partial}{\partial x_2}\widehat{\psi}(\xi)| \leq |h(\xi_1)|\cdot \left(1+\frac{|\xi_2|}{|\xi_1|} \right)^{-\beta}$,
\end{compactenum}
where $\alpha>5$, $\beta \geq 4$, $h\in L^1(\R)$, and $C_1>0$, and suppose that $\psitilde$ satisfies (i) and (ii) with $\xi_1$ and
$\xi_2$ interchanged. Further suppose that $\mathcal{SH}(\phi, \psi, \psitilde,c)$ forms a frame for $L^2(\R^2)$.
Then for all $\nu>0$, any dual frame of the shearlet frame $\mathcal{SH}(\phi, \psi, \psitilde,c)$ provides \emph{(almost) optimally sparse approximations
of functions $f\in \mathcal{E}^2(\nu)$}, i.e.
\begin{align*}
\sum_{n\geq N} \theta_n(f) \lesssim N^{-2} \log(N)^3, \text{ as } N\to \infty,
\end{align*}
where $\theta_n(f)$ denotes the $n$-th entry of the non-increasing rearrangement of the coefficient sequence
$(|\langle f, \psi_{j,k,m,\iota}\rangle_{L^2(\R^2)} |^2)_{(j,k,m,\iota) \in \Lambda}$.
\end{theorem}

We wish to emphasize that all presented results of shearlet systems in this section only hold for $\R^2$. The main objective of
this paper is to introduce a suitable multiscale anisotropic directional system which allows transferring these results to
preferably any bounded domain  {$\Omega \subseteq \R^2$}.

\section{Domain adapted multiscale anisotropic directional systems}\label{subsec:NewConstruction}

The key idea for the construction of domain-adapted multiscale anisotropic directional systems on a domain $\Omega$
which we will coin {\it boundary shearlet systems} consists in combining two different frames, i.e., employing hybrid systems.
To be more precise, we will use boundary wavelet elements to handle the boundary $\partial \Omega$, since these
systems are already adapted to such a boundary. In addition, compactly supported shearlet elements will be used for the
interior of $\Omega$ to achieve the desired optimal approximation rates of cartoon-like functions.
To this end we fix a boundary wavelet system $\mathcal{W}$ for $\Omega$ and a shearlet system $\mathcal{SH}(\phi,\psi,\psitilde,c)$ which fulfill the support properties of (W3), (W4), and (S5) of Assumption \ref{ass:Wave} and \ref{ass:Shear}, respectively.

Certainly, those elements of each system used for the hybrid construction have to be carefully selected. This will be done
by defining a tubular region $\Gamma_{\gamma(j)}$ around the boundary with $\gamma(j)$ depending on the scales $j$ and selecting
those wavelet elements non-trivially intersecting these regions for each scale. As for the selection of the elements of a
compactly supported shearlet system, we choose all those the support of which is completely contained inside $\Omega$. It is conceivable
that the tubular region $\Gamma_{\gamma(j)}$ needs to be defined depending on properties of the shearlets since the frame property
of the hybrid systems requires a small cross-localization of elements from both systems. {Due to these considerations, for $r \in \R$, we now define the tubular region
$\Gamma_{r}$ by
\begin{equation*} 
\Gamma_{r} := \{x \in \Omega : d(x,\partial \Omega) < q_{\mathrm{sh}} 2^{-r}\},
\end{equation*}
i.e., as the part of $\Omega$ that has distance less than $q_{\mathrm{sh}} 2^{- r}$ from $\partial \Omega$. This gives rise to the
following definition.

\begin{definition}\label{def:variableWidthShearletSystem}
Let $s\in \N_0$, $\mathcal{SH}(\phi,\psi,\widetilde{\psi},c) = (\psi_{j,k,m,\iota})_{(j,k,m,\iota) \in \Lambda}$ be a shearlet system fulfilling (S5) of Assumption \ref{ass:Shear}, let $\tau>0$ and $t > 0$. Further, let $\mathcal{W}$ be a boundary
wavelet system fulfilling (W3) and (W4) of Assumption $\ref{ass:Wave}$ and set
\begin{align*}
\mathcal{W}_{t,\tau}:&= \{\omega_{j,m,\upsilon} \in \mathcal{W}: (j,m,\upsilon) \in \Theta_{t,\tau} \},
\end{align*}
where
\begin{align*}
\Theta_{t,\tau} :&= \{(j,m,\upsilon) \in \Theta: B_{2^{-j}(q_{\mathrm{w}}^{(0)} + q_{\mathrm{w}}^{(1)})}(m) \cap \Gamma_{\tau(j-t)} \neq \emptyset\}.
\end{align*}
In addition, let
\[
\Lambda_0: = \{ (j,k,m,\iota)\in \Lambda :  \suppp \psi_{j,k,m,\iota} \subset \Omega\}.
\]
Then, the \emph{boundary shearlet system with offsets $t$ and $\tau$}
is defined as
\begin{align*}\mathcal{BSH}_{t,\tau}^s(\Wcal; \phi, \psi, \psitilde, c):= \{ \psi_{j,k,m,\iota}:  (j,k,m,\iota)\in \Lambda_0\} \cup \mathcal{W}_{t,\tau}.
\end{align*}
\end{definition}
This definition of a boundary shearlet system mimics precisely the program we just intuitively described before. The reader should
notice that as $j \to \infty$, the size of the tubular region shrinks accordingly. Concerning the two offsets, the parameter $t$
acts as a shift for the dependence on the scale, whereas the parameter $\tau$ is merely an overall factor controlling the speed of
decay.

Indeed, this choice of wavelets versus shearlets
provides us with a cross localization property, which will be crucial for the proofs of Theorems \ref{thm:FrameProperty}
and \ref{thm:AnalysisCharOfHOmega}. Since the proof is quite lengthy and technical, we postpone it to Appendix \ref{subsec:loc_sw}. We denote $\Lambda_0^c: = \Lambda \setminus \Lambda_0$ and $\Theta_{\tau, t }^c : = \Theta \setminus \Theta_{t,\tau}$.

\begin{proposition}\label{prop:CrossDecaySum}
Let $s\in \N,$ and let $\mathcal{SH}(\phi, \psi, \psitilde, c)$ and $\mathcal{W}$ satisfy (S2),(S5) and (W3),(W4), respectively. Further let 
$\alpha_{\mathrm{w}} > 1$, $\alpha_{\mathrm{sh}} >0,$ $\beta_{\mathrm{w}} \geq s,$ $\beta_{\mathrm{sh}}>2+2s+2\alpha_{\mathrm{w}}$, $\tau > 0$ and $\epsilon >0$ such that $((1-\epsilon)/\tau - 2)\alpha_{\mathrm{w}} >\frac{5}{2}.$ Then we have for all $t> 0$
\begin{align*}
    \sum_{(j_{\mathrm{sh}},k,m_{\mathrm{sh}},\iota) \in \Lambda_0^c}  \sum_{(j_{\mathrm{w}},m_{\mathrm{w}},\upsilon)\in \Theta_{\tau, t }^c}|\langle (2^{-j_{\mathrm{w}} s}\omega_{j_{\mathrm{w}},m_{\mathrm{w}},\upsilon})^d,2^{-j_{\mathrm{sh}} s} \psi_{j_{\mathrm{sh}},k,m_{\mathrm{sh}},\iota} \rangle_{H^s(\Omega)}|^2 \lesssim   2^{-2(1-\epsilon)\alpha_{\mathrm{w}} t}.
\end{align*}


\end{proposition}

Finally, we note, that it is certainly conceivable to combine other dictionaries in a similar way as described in Definition \ref{def:variableWidthShearletSystem}. For example, anisotropic wavelets are superior to isotropic wavelets in resolving anisotropies along the boundary \cite{AnisotropicWavelets1, AnisotropicWavelets2}. Hence, a hybrid system comprised of anisotropic wavelets and shearlets could yield an even better adapted system if such functions need to be represented. However, we cannot discuss such constructions in this work, since our results crucially depend on the cross-localization of Proposition \ref{prop:CrossDecaySum} and the proof depends crucially on the isotropy of the wavelet elements.


\subsection{Frame property of boundary shearlet systems}\label{subsec:FrameProp}

We now turn to analyzing the frame properties of boundary shearlet systems for $H^s(\Omega)$ on certain domains $\Omega\subset \R^2$.  Towards this goal, the existence of a shearlet system which, after reweighting, forms a frame for $H^s(\R^2)$ (i.e. (S4)), is required. The following result gives an easily satisfiable sufficient condition imposed on the generator functions of the shearlet system, which yields the desired property. We postponed the technical proof to Appendix \ref{sec:ProofHsR2}.

\begin{theorem}\label{Thm:HsR2Frame}
For $s\in\N$ and $\alpha_{\mathrm{sh}} - s> \beta_{\mathrm{sh}} > 4$ let $\phi^1, \psi^1 \in L^2(\R)$ be such that for all $r_1,r_2 \leq 2s$, 
\begin{align*}
    \widehat{D^{r_1}\phi^1}(\xi) &\lesssim \min \{1,|\xi|^{-\beta_{\mathrm{sh}}} \}, \text{ for all } \xi \in \R, \nonumber\\
    \widehat{D^{r_2}\psi^1}(\xi) &\lesssim \min \{1,|\xi|^{\alpha_{\mathrm{sh}}} \} \min \{1,|\xi|^{-\beta_{\mathrm{sh}}} \}, \text{ for all } \xi \in \R.
\end{align*}
We further denote $\phi = \phi^1 \otimes \phi^1,$ $\psi = \psi^1 \otimes \phi^1$ and $\tilde{\psi}(x_1, x_2) := \psi(x_2, x_1)$ for all $x\in \R^2$. Furthermore, we denote for all $\xi \in \R^2$ 
\begin{align}\label{eq:thenotation}
    \widehat{\theta}(\xi) := \xi_1^s \widehat{\psi}(\xi),\qquad    \widehat{\widetilde{\theta}}(\xi) := \xi_2^s \widehat{\widetilde{\psi}}(\xi), \qquad
    \widehat{\mu}(\xi):=  \widehat{\phi}(\xi).
    \end{align}
Assume that there exists $\bar c>0$ such that for all $c_1,c_2\leq \bar{c}$ we have that $\mathcal{SH}(\mu,\theta,\tilde{\theta},(c_1,c_2))$ forms a frame for $L^2( \mathbb R^2)$ with a lower frame bound independent of $c_1,c_2$. 
Then there exists $\tilde{c}>0$ such that for all $c_1=c_2 \leq \tilde{c}$ the system $ = \mathcal{SH}(\phi,\psi,\psitilde,c),$ where $c=(c_1,c_2)$ obeys
$$
\|f\|_{H^s(\mathbb{R}^2)}^2\sim  \sum_{(j,k,m,\iota)\in \Lambda} |\langle f, 2^{-js} \psi_{j,k,m,\iota} \rangle_{H^s(\R^2)}  |^2,
$$
for all $f\in H^s(\R^2)$.
\end{theorem}

Having established the existence of shearlet frames for $H^s(\R^2),$ we can now turn our attention to the analysis of the frame property of boundary shearlet systems for $H^s(\Omega)$.  We require $\Omega$ to be chosen in such a way that there exists some linear extension operator $\ExtOpMod: H^s(\Omega)\to H^s(\R^2)$ such that
$\ExtOpMod(f)_{|\Omega} = f$ for all $f\in H^s(\Omega)$ and $\|\ExtOpMod\|_{H^s(\Omega) \to H^s(\R^2)} \leq {M_{\mathrm{ext}}}$ for some ${M_{\mathrm{ext}}>0}$. For example, domains which are \emph{minimally smooth} (see \cite{Stein1970SingIng}) fulfill this property. The following result now gives us weak sufficient conditions, under which a boundary shearlet system indeed forms a frame.

\begin{theorem} \label{thm:FrameProperty}
Let $\mathcal{SH}(\phi, \psi, \psitilde, c)$ and $\mathcal{W}$ satisfy (S2), (S4), (S5) and (W2), (W3), (W4), respectively with $s\in\N\cup\{0\},$ $\alpha_{\mathrm{w}}, \alpha_{\mathrm{sh}},\beta_{\mathrm{w}},\beta_{\mathrm{sh}},\tau,\epsilon$ as in Proposition \ref{prop:CrossDecaySum}. 
Then there exists some $T>0$ such
that, for any $t\geq T$, the boundary shearlet system ${(\varphi_n)_{n\in \N} = \mathcal{BSH}_{t,\tau}^s(\bdWaveS; \phi, \psi, \psitilde, c)}$ satisfies
\begin{align}\label{eq:FrameForHs}
\|f\|_{H^s(\Omega)}^2  \sim \sum_{n\in \N} |\left \langle f, 2^{-j_n s}\varphi_{n} \right \rangle_{H^s(\Omega)} |^2 \fa f\in H^s(\Omega). 
\end{align}
\end{theorem}

\begin{proof}
By the definition of $\mathcal{BSH}_{t,\tau}^s(\bdWaveS; \phi, \psi, \psitilde, c)$ we have that $\suppp (2^{-js}\omega_{j,m,\nu})^d\cap \partial \Omega = \emptyset$ for all $(j,m,\upsilon) \in \Theta_{t, \tau}^c$. As a consequence, we have that 
\begin{align}\label{eq:extOfInternals}
\|f^{(\mathrm{int})}\|_{H^s(\Omega)} &= \|\tilde{f}^{(\mathrm{int})}\|_{H^s(\R^2)} \fa f\in H^s(\Omega), \text{ where }\\
f^{(\mathrm{int})} &= \sum_{(j,m,\upsilon) \in \Theta_{t, \tau}^c} \langle f, 2^{-js} \omega_{j,m,\upsilon} \rangle_{H^s(\Omega)}(2^{-js}\omega_{j,m,\upsilon})^d \nonumber
\end{align}
and $\tilde{f}^{(\mathrm{int})}$ denotes the trivial extension by $0$ of $f^{(\mathrm{int})}$ to $\R^2$. We will invoke this projection operator further below in the proof. We proceed by observing that there exist
constants $0<A_{\mathrm{sh}} \leq B_{\mathrm{sh}} <\infty$ such that
\begin{align*}
A_{\mathrm{sh}} \|f\|_{H^s(\R^2)}^2  \leq \sum_{(j,k,m,\iota) \in \Lambda} |\left \langle f, 2^{-j s}\psi_{j,k,m,\iota}\right \rangle_{H^s(\R^2)} |^2 \leq B_{\mathrm{sh}}\|f\|_{H^s(\R^2)}^2 \fa f\in H^s(\R^2).
\end{align*}
Moreover, the wavelet system $\mathcal{W}$ obeys
\begin{align*}
C_{\mathrm{w}} \|f\|_{H^s(\Omega)}^2  \leq \sum_{(j,m,\upsilon) \in \Theta} |\left \langle f, 2^{-js}{\omega_{j,m,\upsilon}}\right \rangle_{H^s(\Omega)} |^2 \leq D_{\mathrm{w}}\|f\|_{H^s(\Omega)}^2 \fa f\in H^s(\Omega),
\end{align*}
for some $0<C_{\mathrm{w}} \leq D_{\mathrm{w}} <\infty$.

We start by proving the upper bound in \eqref{eq:FrameForHs}. After setting $(\varphi_n)_{n\in \N}:= {\mathcal{BSH}_{t,\tau}^s(\mathcal{W}; \phi, \psi, \psitilde, c)}$ we obtain that
\begin{align*}
\sum_{n\in \N}& |\left \langle f, 2^{-j_n s}\varphi_n\right \rangle_{H^s(\Omega)} |^2 \\
&= \sum_{(j,k,m,\iota) \in \Lambda_0} |\left \langle f, 2^{-j s}\psi_{j,k,m,\iota}\right \rangle_{H^s(\Omega)} |^2 + \sum_{{(j,m,\upsilon) \in \Theta_{t,\tau}}} |\left \langle f, 2^{-j s}{\omega_{j,m,\upsilon}}\right \rangle_{H^s(\Omega)} |^2 \\
&= \sum_{(j,k,m,\iota) \in \Lambda_0} |\left \langle f, 2^{-j s}\psi_{j,k,m,\iota}\right \rangle_{H^s(\R^2)} |^2 + \sum_{{(j,m,\upsilon) \in \Theta_{t,\tau}}} |\left \langle f, 2^{-j s}{\omega_{j,m,\upsilon}}\right \rangle_{H^s(\Omega)} |^2 \\
&\leq (B_{\mathrm{sh}} + D_{\mathrm{w}}) \| f\|_{H^s(\Omega)}^2,
\end{align*}
where we used that 
$$
\left \langle f, 2^{-j s}\psi_{j,k,m,\iota}\right \rangle_{H^s(\Omega)} = \left \langle f, 2^{-j s}\psi_{j,k,m,\iota}\right \rangle_{H^s(\R^2)}
$$
holds, since $\suppp \psi_{j,k,m,\iota} \subset \Omega$ for all $(j,k,m,\iota) \in \Lambda_0$.
We have just obtained the existence of an upper bound.

For the existence of the lower bound, let $f^{(\mathrm{bd})}:=f-f^{(\mathrm{int})}.$ 
Due to the boundedness of the synthesis operator of the frame $(2^{-js}\omega_{j,m,\upsilon})^d_{(j,m,\upsilon) \in \Theta}$ we have that $\| f \|_{H^s(\Omega)} \sim \| f^{(\mathrm{int})} \|_{H^s(\Omega)} + \| f^{(\mathrm{bd})} \|_{H^s(\Omega)}$ for all $f \in H^s(\Omega)$. 
By assumption on $\Omega$ there exists a bounded linear extension operator $\ExtOpMod: H^s(\Omega)\to H^s(\R^2)$ such that
$\ExtOpMod(f)_{|\Omega} = f$ for all $f\in H^s(\Omega)$ and ${\|\ExtOpMod\|_{H^s(\Omega) \to H^s(\R^2)} \leq {M_{\mathrm{ext}}}}$ for some ${M_{\mathrm{ext}}>0}$.

Based on this, we define the
following modified extension operator
\begin{align*}
\ExtOp: H^s(\Omega) &\to H^s(\R^2), \qquad f \mapsto \left\{\begin{array}{l l}
\ExtOpMod(f^{(\mathrm{bd})}) + f^{(\mathrm{int})} & \text{ on } \Omega, \\
\ExtOpMod(f^{(\mathrm{bd})}) & \text{ on } \R^2 \setminus \Omega. \\
\end{array}\right.
\end{align*}
{We first notice that $\ExtOp (f)_{| \Omega} = f$. {To obtain the well-definedness of $\ExtOp$ we continue by proving the boundedness of $\ExtOp$}. 

By invoking \eqref{eq:extOfInternals} we obtain the boundedness of the operator $\ExtOp$ from the following estimate:
\begin{align*}
\| \ExtOp (f) \|_{H^s(\R^2)}^2 &\lesssim \| E(f^{(\mathrm{bd})})\|_{H^s(\R^2)}^2 + \| f^{(\mathrm{int})} \|_{H^s(\Omega)}^2 \\
&\lesssim \| f^{(\mathrm{bd})}\|_{H^s(\Omega)}^2 + \| f^{(\mathrm{int})}\|_{H^s(\Omega)}^2 \lesssim \| f \|_{H^s(\Omega)}^2,
\end{align*}
where we have used the boundedness of $\ExtOpMod$ in the second inequality.}

Using the operator $\ExtOp(f)$, we obtain that
\begin{align}
A_{\mathrm{sh}}  \| f \|_{H^s(\Omega)}^2 \leq \ & A_{\mathrm{sh}} \| \ExtOp(f) \|_{H^s(\R^2)}^2 \nonumber\\
\leq \ & \sum_{(j,k,m,\iota) \in \Lambda} |\langle \ExtOp(f), 2^{-js}\psi_{j,k,m,\iota}\rangle_{H^s(\R^2)} |^2 \nonumber\\
 = \ & \sum_{(j,k,m,\iota) \in \Lambda_0} |\left \langle f, 2^{-js}\psi_{j,k,m,\iota}\right \rangle_{H^s(\Omega)} |^2 \nonumber\\
 &\quad +\sum_{(j,k,m,\iota) \in \Lambda_0^c} |\langle \ExtOp(f), 2^{-js}\psi_{j,k,m,\iota}\rangle_{H^s(\R^2)} |^2 \nonumber\\
 \leq \ & \sum_{(j,k,m,\iota) \in \Lambda_0} |\left \langle f, 2^{-js}\psi_{j,k,m,\iota}\right \rangle_{H^s(\Omega)} |^2 +\nonumber \\
& \quad+ 2\Bigg(\sum_{(j,k,m,\iota) \in \Lambda_0^c} |\langle \ExtOp(f)-\ExtOp(f^{(\mathrm{bd})}), 2^{-js}\psi_{j,k,m,\iota}\rangle_{H^s(\R^2)} |^2 \nonumber\\
&\qquad +\sum_{(j,k,m,\iota) \in \Lambda_0^c} |\langle \ExtOp(f^{(\mathrm{bd})}), 2^{-js}\psi_{j,k,m,\iota}\rangle_{H^s(\R^2)} |^2\Bigg)\nonumber\\
 = \ & \sum_{(j,k,m,\iota) \in \Lambda_0}|\left \langle f, 2^{-js}\psi_{j,k,m,\iota}\right \rangle_{H^s(\Omega)} |^2 + 2\left(\mathrm{I} + \mathrm{II}\right). \label{eq:theEquationAbove}
\end{align}
We next estimate $\mathrm{I}$ and $\mathrm{II}$, starting with $\mathrm{II}$. By using the frame property for $H^s(\Omega)$ of $(2^{-js}\omega_{j,m,\upsilon})_{(j,m,\upsilon) \in \Theta}$, we immediately
obtain the required estimate
\begin{align*}
\mathrm{II}\leq B_{\mathrm{sh}}\|\ExtOp(f^{(\mathrm{bd})})\|^2_{H^s(\R^2)} \leq  M_{\mathrm{ext}}^2 B_{\mathrm{sh}}\|f^{(\mathrm{bd})}\|^2_{H^s(\Omega)} \leq \frac{M_{\mathrm{ext}}^2 B_{\mathrm{sh}}}{C_{\mathrm{w}}} \sum_{{(j,m,\upsilon) \in \Theta_{\tau, t }}} |\left \langle f, 2^{-js}\omega_{j,m,\upsilon}\right \rangle_{H^s(\Omega)} |^2,
\end{align*}
where we used that the synthesis operator of the dual frame of $(2^{-js}\omega_{j,m,\upsilon})^d_{(j,m,
\upsilon) \in \Theta}$ is bounded in operator norm by $\frac{1}{\sqrt{C_{\mathrm{w}}}}$.
The existence of a positive lower bound follows by subtracting {$2\mathrm{I}$ on both sides of the inequality \eqref{eq:theEquationAbove}}, provided that
we can show
\begin{align} \label{eq:delta}
\mathrm{I} < A_{\mathrm{sh}}/2 \|f\|_{H^s(\Omega)}^2.
\end{align}
A lower frame bound is then given by
\begin{align*}
(A_{\mathrm{sh}} - 2 \mathrm{I})C_{\mathrm{w}} / (2 M_{\mathrm{ext}}^2 B_{\mathrm{sh}}).
\end{align*}
Since by construction, $\ExtOp(f) - \ExtOp(f^{(\mathrm{bd})})  = \ExtOp(f-f^{(\mathrm{bd})}) = \ExtOp(f^{(\mathrm{int})})$, we can compute
{\allowdisplaybreaks
\begin{align*}
\mathrm{I} =& \sum_{(j,k,m,\iota) \in \Lambda_0^c} \left|\left \langle \ExtOp(f)-\ExtOp(f^{(\mathrm{bd})}), 2^{-js}\psi_{j,k,m,\iota}\right \rangle_{H^s(\R^2)} \right|^2 \\
=& \sum_{(j,k,m,\iota) \in \Lambda_0^c} 2^{-2js}\left|\left \langle \ExtOp(f^{(\mathrm{int})}), \psi_{j,k,m,\iota}\right \rangle_{H^s(\R^2)} \right|^2\\
=& \sum_{(j,k,m,\iota) \in \Lambda_0^c} 2^{-2js}\left|\left \langle f^{(\mathrm{int})}, \psi_{j,k,m,\iota}\right \rangle_{H^s(\R^2)} \right|^2\\
 = & \ \sum_{(j,k,m,\iota) \in \Lambda_0^c} \left| 2^{-js} \sum_{ {(j',m',\upsilon) \in \Theta_{t,\tau}^c}}   \left \langle f, {2^{-j' s}\omega_{j',m',\upsilon}} \right \rangle_{H^s(\Omega)}  \left \langle ({2^{-j' s} \omega_{j',m',\upsilon}})^d, \psi_{j,k,m,\iota}\right \rangle_{H^s(\R^2)} \right|^2\\
  = & \ \sum_{(j,k,m,\iota) \in \Lambda_0^c} \left|\sum_{ {(j',m',\upsilon) \in \Theta_{t,\tau}^c}} \left \langle f, {2^{-j' s}\omega_{j',m',\upsilon}} \right \rangle_{H^s(\Omega)}   \left \langle {(2^{-j's}\omega_{j',m',\upsilon})^d}, 2^{-j s}\psi_{j,k,m,\iota}\right \rangle_{H^s(\R^2)} \right|^2.
\end{align*}
}
Applying the Cauchy-Schwarz inequality then yields
\begin{align}
\mathrm{I} \leq & \hspace*{-0.4cm}
\sum_{(j,k,m,\iota) \in \Lambda_0^c}\left( \sum_{(j',m',\upsilon) \in \Theta_{t,\tau}^c} \left|\left \langle f, {2^{-j s}\omega_{j',m',\upsilon}} \right \rangle_{H^s(\Omega)}\right|^2\right)\cdot \nonumber \\
& \quad \cdot \left(  \sum_{ {(j',m',\upsilon) \in \Theta_{t,\tau}^c}} \left|\left \langle {(2^{-j's}\omega_{j',m',\upsilon})^d},2^{-j s} \psi_{j,k,m,\iota}\right \rangle_{H^s(\Omega)} \right|^2\right) \nonumber\\
{\leq} & D_{\mathrm{w}} \|f\|_{H^s(\Omega)}^2 \sum_{(j,k,m,\iota) \in \Lambda_0^c}  \sum_{(j',m',\upsilon) \in \Theta_{t,\tau}^c} \left|\left \langle (2^{-j's}{\omega_{j',m',\upsilon}})^d, 2^{-j s}\psi_{j,k,m,\iota}\right \rangle_{H^s(\Omega)} \right|^2. \label{eq:I}
\end{align}
Furthermore, by Proposition \ref{prop:CrossDecaySum} we have that for any given $\epsilon >0$ there exists a sufficiently large offset $t$ such that
\begin{align} \label{eq:I2}
\sum_{(j,k,m,\iota) \in \Lambda_0^c} \sum_{(j',m',\upsilon) \in \Theta_{t,\tau}^c}   \left|\left \langle (2^{-j' s}{\omega_{j',m',\upsilon}})^d, 2^{-j s}\psi_{j,k,m,\iota}\right \rangle_{H^s(\Omega)} \right|^2 < \epsilon.
\end{align}
Applying \eqref{eq:I2} to \eqref{eq:I} proves \eqref{eq:delta}, thereby completing the proof. 
\end{proof}


\section{Stability of boundary shearlet systems}\label{sec:caracH}

In this section we will show that the analysis and synthesis operators of boundary shearlets provide stable maps between weighted coefficient sequences and $H^s(\Omega)$. First of all, we are interested in providing a characterization of Sobolev spaces $H^s(\Omega)$, $s \geq 0$. In other words, we aim to show that there exist $0<A\leq B<\infty$ such that
\begin{align}\label{eq:CharStart}
A \|f \|^2_{H^s(\Omega)} \leq \sum_{n\in \N} 2^{2 j_n s}| \langle f, \varphi_n\rangle_{L^2(\Omega)} |^2 \leq B \|f \|_{H^s(\Omega)}^2 \quad \text{for all } f\in H^s(\Omega),
\end{align}
where $\mathcal{BSH}_{t,\tau}^s(W;\phi,\psi,\psitilde,c)=:(\varphi_n)_{n\in\N}$ and $j_n$ corresponds to the scale of the underlying element, i.e.,  of $\varphi_n=\psi_{j_n,k_n,m_n,\iota_n}$ or $\varphi_n=\omega_{j_n,m_n,\upsilon_n}.$ 
The characterization \eqref{eq:CharStart} states that the analysis operator of a bounded shearlet system is bounded from above and below as a map from $H^s(\Omega)$ to the weighted sequence space $\ell^{2,s}$ where 
\begin{align*}
\|c\|_{2,s}:= \|(2^{j_n s}c_{n})_{n\in \N} \|_{\ell^2}, \text{ and }\ell_{2,s} := \{ c\in \ell^2: \|c\|_{2,s}<\infty\}.
\end{align*}
We will demonstrate Property \eqref{eq:CharStart} in Subsection \ref{subsec:Hsomega}.
Another mapping property, that is of considerable interest for the discretization of PDEs is the Gelfand frame property. This property requires a bounded synthesis operator as a map from $\ell^{2,s}$ to $H^s(\Omega)$ and a bounded analysis operator with respect to the $L^2$ dual frame as a map from $H^s(\Omega)$ to $\ell^{2,s}$. We will discuss this property in detail in Subsection \ref{sec:GelfProp}.

\subsection{Characterization of $H^s(\Omega)$} \label{subsec:Hsomega}

In order to obtain the characterization result of \eqref{eq:CharStart} we will require a similar property for shearlets on $\R^2$, i.e., property (S3) for $s \geq 0$. In fact, easily realizable conditions can be given that guarantee (S3) for a shearlet system. We remark that the next result, is only stated for $s\in\N,$ in \cite{DissPP}, but also holds for any $s\geq 0$ with the same proof.

\begin{theorem}\cite{DissPP} \label{thm:charThm}
Let $s\geq 0$, and $\phi, \psi\in \Lp{2}{(\R^2)}$ such that for all $\xi \in \R^2$
\begin{align*}
|\widehat \phi(\xi)| &\lesssim \min\{1, |\xi_1 |^{-\beta}\}\min\{1, |\xi_2 |^{-\beta}\}\nonumber,\\
|\widehat \psi(\xi)| &\lesssim \min\{1, |\xi_1 |^\alpha\}\min\{1, |\xi_1 |^{-\beta}\}\min\{1, |\xi_2|^{-\beta} \},
\end{align*}
with $\beta>4$ and $\alpha > \beta + s$ and let $\tilde{\psi}(x_1,x_2): = \psi(x_2,x_1)$ for all $x\in \R^2$. Further define 
\begin{align*}
\rho: \R^2 \to \R, \quad
\rho(\xi) := (1+|\xi_1|^2)^{\frac{1}{2}}(1+|\xi_2|^2)^{\frac{1}{2}} 
\end{align*}
and 
\begin{align*}
    \hat{\theta}(\xi) := \rho(\xi)^{-s} \widehat{\psi}(\xi),\qquad    \hat{\tilde{\theta}}(\xi) := \rho(\xi)^{-s} \widehat{\widetilde{\psi}}(\xi), \qquad
    \widehat{\mu}(\xi):= \rho(\xi)^{-s} \widehat{\varphi}(\xi), 
\end{align*}
for all $\xi \in \R^2$. Assume that there exists $\tilde{c}>0$ such that $\mathcal{SH}(\mu,\theta,\tilde{\theta},(c_1,c_2))$ forms a frame for $\Lp{2}{( \mathbb R^2)}$ for all $c_1,c_2\leq \tilde{c}$ with lower frame bound independent of $c_1,c_2$.
Then there exists $\bar{c}$ such that for all $c\leq \bar{c}$ we have that the system $(\psi_{j,k,m,\iota})_{(j,k,m,\iota) \in \Lambda}=\mathcal{SH}(\phi,\psi,\psitilde,(c,c))$ satisfies (S3).
\end{theorem}

We now use the results on $\R^2$ to obtain a characterization of $H^s(\Omega)$ by the analysis coefficients with respect to
a shearlet system on $\Omega$. At the end of this section, we briefly remark on a characterization of $H^s(\Omega)$ by dual frame coefficients, with a bit more elaborate discussion in Subsection \ref{sec:GelfProp}.

The following theorem states the main result of this subsection. 

\begin{theorem}\label{thm:AnalysisCharOfHOmega}
Let $s\geq 0, \alpha_{\mathrm{w}} > 1$, $\beta_{\mathrm{w}} > \alpha_{\mathrm{w}}+1$, $\beta_{\mathrm{sh}} > 1 + \alpha_{\mathrm{w}}$, $\tau > 0$ and $\epsilon >0$ such that ${((1-\epsilon)/\tau - 2)\alpha > 5/2.}$ Let further 
\begin{itemize}
\item $\mathcal{SH}(\phi, \psi, \psitilde, c)$ be a shearlet system satisfying (S1), (S3), (S3), and (S5) and
\item $\mathcal{W}$ be a boundary wavelet system satisfying (W1), (W2), the condition (W3) for $s=0$ and (W4)
\end{itemize}
with $\alpha_{\mathrm{sh}}, \beta_{\mathrm{sh}},\alpha_{\mathrm{w}}, \beta_{\mathrm{w}}$ as stipulated above. Then there exists some $T>0$ such that, for any $t\geq T$, there exist $0<A\leq B<\infty$
and the boundary shearlet system $(\varphi_n)_{n\in \N} = \mathcal{BSH}_{t,\tau}^0(\Wcal; \phi, \psi, \psitilde, c)$ satisfies \eqref{eq:CharStart}.
\end{theorem}

\begin{proof}
The proof is very similar to that of Theorem \ref{Thm:HsR2Frame} and is thus omitted here. It can be found in \cite[Theorem 4.2.6]{DissPP}.
\end{proof}

\subsection{Gelfand frame property}\label{sec:GelfProp}

In this subsection we will analyze a property that allows us to use boundary shearlet systems to adaptively solve elliptic partial differential equations.
In \cite{dahlke2007adaptive} it was observed that a highly beneficial property for the discretization of elliptic PDEs is the property of constituting a \emph{Gelfand frame}.
Let $s\in \N$ and let us consider the Gelfand triple $(H^s(\Omega), L^2(\Omega), H^{-s}(\Omega))$. Then, a Gelfand frame is a frame $(\varphi_n)_{n\in \N}$ for
$L^2(\Omega)$ satisfying the following two properties:
\begin{align*}
\left\|\sum_{n\in \N} c_n \varphi_n\right\|_{H^s(\Omega)}^2 \lesssim \left\|(2^{j_n s} c_n)_{n\in \N}  \right\|_{\ell^2}^2 \fa c \in \ell^2 \tag{GFA1}
\end{align*}
 and
\begin{align*}
\left\|(2^{j_n s}  \langle f, \varphi_n^d \rangle_{L^2(\Omega)})_{n\in \N}\right\|_{\ell^2}^2 \lesssim \|f\|_{H^s(\Omega)}^2 \fa f \in H^s(\Omega), \tag{GFA2}
\end{align*}
where the implicit constants are independent of $c$ and $f$, respectively.
We wish to mention, that the Gelfand frame property is a stronger requirement than
the characterization of $H^s(\Omega)$ guaranteed by Theorem \ref{thm:AnalysisCharOfHOmega}. 

We will demonstrate that the property (GFA1) can be easily obtained from (W3) and (S3) if the wavelet and shearlet frames are sufficiently \emph{localized}. We define the Gramian of the wavelet system $\mathcal{W}$ and a shearlet system $\mathcal{SH}(\phi, \psi, \tilde{\psi}, c)$ by
\begin{align*}
    G_{\mathrm{w}} : \  & \ell^2(\Theta) \to \ell^2(\Theta), \ G_{\mathrm{sh}} : \ell^2(\Lambda) \to \ell^2(\Lambda):\\
    (G_{\mathrm{w}}((c_\lambda)_{\lambda \in \Theta}))_\mu & := \sum_{\lambda \in \Theta}  c_\lambda \langle \omega_{\lambda}, \omega_{\mu} \rangle \text{ for } \mu \in \Theta, \\
    (G_{\mathrm{sh}}((c_\lambda)_{\lambda \in \Lambda}))_\mu & := \sum_{\lambda \in \Lambda}  c_\lambda \langle \psi_{\lambda}, \psi_{\mu} \rangle \text{ for }\mu \in \Lambda.
\end{align*}
If the the underlying frames admit a suffiently high localization, i.e., if $\langle \psi_{\lambda}, \psi_{\mu} \rangle$, $\langle \omega_{\lambda}, \omega_{\mu} \rangle$ decay quickly for $\mathrm{dist}(\lambda, \mu)$ growing, then it holds that 
\begin{align} \label{eq:localization}
\|G_{\mathrm{sh}}\|_{\ell^{2,s}(\Lambda) \to \ell^{2,s}(\Lambda)} < \infty\text{ and }\|G_{\mathrm{w}}\|_{\ell^{2,s}(\Theta) \to \ell^{2,s}(\Theta)} < \infty.
\end{align}  
For example, sufficient localization such that \eqref{eq:localization} holds is clearly given for $G_{\mathrm{w}}$ if the wavelet system forms an orthonormal basis. For shearlets, the localization and the mapping property \eqref{eq:localization} was analyzed in \cite[Sec. 3.2]{Gro2013}.

Assuming wavelet and shearlet systems are such that \eqref{eq:localization} is satisfied, we can establish (GFA1) directly. First of all, given $c \in \ell^{2,s}$ we can split the wavelet and shearlet parts by:
\begin{align*}
    \|\sum_{n\in \N} c_n \varphi_n\|_{H^s(\Omega)}^2 \lesssim \|\sum_{(j,m,\upsilon) \in \Theta_{t, \tau}} c_{j,m,\upsilon}^{\mathrm{w}} \omega_{j,m,\nu}\|_{H^s(\Omega)}^2 + \|\sum_{(j,k, m,\iota) \in \Lambda_0} c_{j,k, m,\iota }^{\mathrm{sh}} \psi_{j,k, m,\iota}\|_{H^s(\Omega)}^2.
\end{align*}
Invoking (S3) yields that 
\begin{align*}
&\left\|\sum_{(j,k,m,\iota)\in \Lambda} c_{j,k,m,\iota}^{\mathrm{sh}} \psi_{j,k,m,\iota} \right\|_{H^s(\Omega)}\\
\lesssim &\quad \left\| (\langle \sum_{(j,k,m,\iota)\in \Lambda} c_{j,k,m,\iota}^{\mathrm{sh}} \psi_{j,k,m,\iota}, \psi_{j',k',m',\iota'} \rangle_{L^2(\R^2)})_{(j',k',m',\iota')\in \Lambda} \right\|_{\ell^{2,s}(\Lambda)}\\
= &\| G_{\mathrm{sh}}(c^{\mathrm{sh}}) \|_{\ell^{2,s}(\Lambda)} \leq \| G_{\mathrm{sh}} \|_{\ell^{2,s}(\Lambda)\to \ell^{2,s}(\Lambda)} \|c^{\mathrm{sh}}\|_{\ell^{2,s}(\Lambda)}.
\end{align*}
A similar estimate for the wavelet part in addition to the fact that 
$$
\|c^{\mathrm{sh}}\|_{\ell^{2,s}(\Lambda)}^2 + \|c^{\mathrm{w}}\|_{\ell^{2,s}(\Lambda)}^2 = \|c\|_{\ell^{2,s}}^2
$$
yields (GFA1).

Unfortunately, the second property, (GFA2), is less accessible than (GFA1).
This is the case, since (GFA2) concerns an estimate on dual frame coefficients.
An explicit construction of any dual frame is, however, not available in our setting.
A concrete construction of a dual is not even known for the standard shearlet systems from Subsection \ref{sec:shearlets}. The only
known first construction can be found in \cite{KL15}. However, the resulting system has a different structure than a standard shearlet
system and is, in particular, highly redundant. Hence it is not clear how to even obtain characterizations of Sobolev spaces with the primal frame of the system of \cite{KL15}.} Therefore presenting such a characterization with
dual frame coefficients is beyond the scope of this paper.

To still provide an understanding of (GFA2) and to prove, that it most likely can be satisfied for shearlet systems, we include a numerical analysis of this property in Subsection \ref{sec:GelfPropNum}, where we will see that the shearlet systems of Shearlab possess the property (GFA2).

\section{Approximation properties}\label{sec:ApproxPropOmega}

Finally, we discuss approximation properties of the boundary shearlet systems. In Subsection \ref{sec:ApproxPropPlane} it was discussed
that shearlet systems on $\R^2$ yield optimally sparse approximations of cartoon-like functions. To obtain a similar result for the newly
introduced boundary shearlet systems, we first need to specify what we mean by cartoon-like functions on bounded domains $\Omega \subset \R^2$.

The attentive reader will have noticed that the definition of cartoon-like functions, i.e., functions in $\mathcal{E}^2(\nu)$, already focuses on functions with compact support in
$[0,1]^2$. For our purposes and according to anticipated applications in imaging science and numerical analysis of partial differential equations, this definition is too restrictive. In
fact, it does not include discontinuity curves, which not only touch the boundary of the domain but intersect it, in particular,
producing a point singularity on the boundary. This situation shall now be included.

The following definition makes these thoughts precise and generalizes
the previous notion of cartoon-like functions from Definition \ref{def:CartoonLikeRTwo}, see also Figure \ref{fig:cartoonLikeBdDom} for an illustration.

\begin{definition}\label{def:CartoonOnOmega}
Let $\nu>0$, $\Omega \subset \R^2$, $D \subset \R^2$, and $f = f_1 + \chi_D f_2$ with $f_i \in C^2(\R^2)$ and $\suppp f_i \subset
[-c_{supp}, c_{supp}]^2$ for some $c_{supp}>0$ and $i = 1,2$ such that $f(2c_{supp} \cdot-(1/2,1/2)) \in \mathcal{E}^2(2c_{supp} \cdot \nu)$.
Further, let $\#(\partial D \cap \partial \Omega) \leq M$ for some $M\in \N$  and let $\partial D$ and $\partial \Omega$ only intersect transversely. Then we call $P_{\Omega} f$ a \emph{cartoon-like function on
$\Omega$}, and denote the set of \emph{cartoon-like functions on $\Omega$} by $\mathcal{E}^2(\nu, \Omega)$.
\end{definition}

\begin{figure}[ht!]
\centering
\includegraphics[scale=0.6]{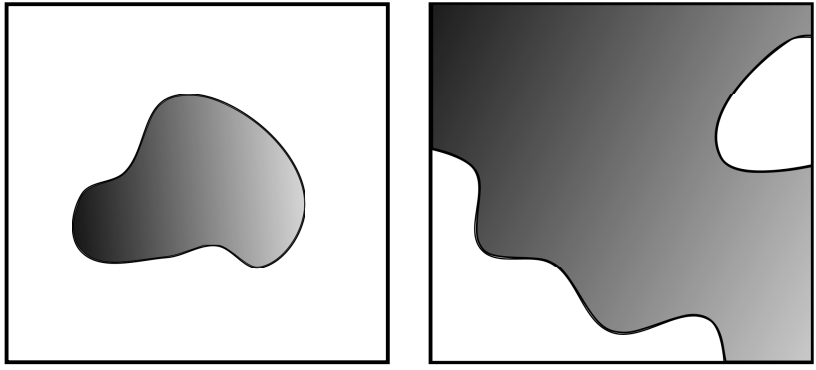}
\caption{Cartoon-like functions in $\mathcal{E}^2(\nu, [0,1]^2)$ for some $\nu >0$.} \label{fig:cartoonLikeBdDom}
\end{figure}

To estimate the error of the best $N$-term approximation, we use
the following well-known approach. Let $\mathcal{BSH}_{t,\tau}^0(\Wcal; \phi, \psi, \psitilde, c) =: (\varphi_n)_n$ be a boundary shearlet system which forms a frame for $L^2(\Omega)$, let
$f \in \mathcal{E}^2(\nu, \Omega)$, and let $(\theta_n(f))_{n\in \N}$ be the non-increasing rearrangement of $(|\left \langle
\varphi_n, f\right \rangle_{L^2(\Omega)}|^2)_{n\in \N}$. Then, by the frame inequality, we have
\begin{align*}
\|f - f_N\|_2^2 \lesssim \sum_{n\geq N} \theta_n(f) \quad \mbox{for all } N\in \N.
\end{align*}
Hence to obtain the optimal best $N$-term approximation rate of Theorem \ref{thm:compShearAreOptSparse}, we require the estimate
$$
\sum_{n\geq N} \theta_n(f) \lesssim N^{-2} \log(N)^3 \text{ as } N\to \infty.
$$
The reader will certainly have noticed that it is not initially clear that this is also the optimally possible rate for
the extended class of cartoon-like functions on $\Omega$ from Definition \ref{def:CartoonOnOmega}. But it is easy to see that a best $N-$term
approximation rate {\it faster} than $N^{-2}$ violates the optimality result of \cite{DCartoonLikeImages2001} for functions in
$\mathcal{E}^2(\nu)$, since each system on $[0,1]^2$ can be extended by $0$ to yield a system on $\R^2$. Recall that functions in
$\mathcal{E}^2(\nu)$ vanish outside $[0,1]^2$, hence the extended system would imply a faster than optimal approximation rate for
functions in $\mathcal{E}^2(\nu)$, a contradiction. Thus, the best $N$-term approximation rate is bounded from below by $N^{-2}$, but it is not clear whether this rate can actually be achieved. Theorem \ref{thm:OptSpaApproxOnBdDom} shows that this is indeed the
case up to a negligible log factor.

To proceed we need to first assemble some results about the approximation rates of wavelets and shearlets. The optimal approximation
rate of shearlets is guaranteed by Theorem \ref{thm:compShearAreOptSparse}. Moreover, let $(\omega_{j,m,\upsilon})_{(j,m,\upsilon)\in \Theta}$ be a wavelet system admitting (W3) with $s = 2$, i.e., 
\begin{align*}
\|f\|_{H^{2}(\Omega)}^2 \sim \sum_{j,m} 2^{4j}|\left \langle f, \omega_{j,m,\upsilon}\right \rangle_{L^2(\Omega)}|^2 \quad \text{ for all } f \in H^{2}(\Omega).
\end{align*}
Then, the result \cite[Thm. 39.2]{Coh2000} implies that, for any $f\in H^2(\Omega)$ and 
\begin{align}\label{eq:estForWavelets1}
\sum_{n\geq N} \theta_n^\omega(f) \lesssim N^{-2},
\end{align}
where $(\theta_n^{\omega}(f)_n)_{n\in \N}$ denotes a non-increasing
rearrangement of $(|\langle f, \omega_{j,m,\upsilon}\rangle_{L^2(\Omega)}|^2)_{(j,m,\upsilon)\in \Theta}.$

Based on these approximation results, we obtain the following theorem for the approximation rate of cartoon-like functions on $\Omega$
by the hybrid system of boundary shearlets.

\begin{theorem}\label{thm:OptSpaApproxOnBdDom}
Let $\Omega \subset \R^2$ with $\partial \Omega$ having finite length. Further, let $\phi, \psi, \psitilde$ fulfill the assumptions of Theorem \ref{thm:compShearAreOptSparse}, and
let $\mathcal{W}$ be a wavelet basis satisfying (W4) and \eqref{eq:estForWavelets1}. Further let $t > 0$, $\tau > 1/3$, and let $\mathcal{BSH}_{t,\tau}^0(\Wcal; \phi, \psi, \psitilde, c) =: (\varphi_n)_{n\in\N}$ be a boundary shearlet
system, which forms a frame for $L^2(\Omega)$. Then $\mathcal{BSH}_{t,\tau}^0(\Wcal; \phi, \psi, \psitilde, c)$ yields almost optimally sparse approximation for cartoon-like
functions on $\Omega$, i.e., for all $f\in \mathcal{E}^2(\nu,\Omega)$,
$$
\|f - f_N\|_{L^2(\Omega)}^2 \lesssim N^{-2}\log(N)^3 \quad \text{for } N\to \infty,
$$
where $f_N = \sum_{n\in I_n} \left\langle f, \varphi_n\right \rangle_{L^2(\Omega)} \varphi_n^d$ with $I_n$ containing the $N$ largest coefficients
$\left\langle f, \varphi_n\right \rangle_{L^2(\Omega)}$ in modulus and $(\varphi_n^d)_{n\in \N}$ is any dual frame of $(\varphi_n)_{n\in \N}$.
\end{theorem}

\begin{proof}
Let $f= P_{\Omega}(f_1 + \chi_D f_2) \in \mathcal{E}^2(\nu, \Omega)$ and consider
\begin{align} \label{eq:approx}
\|f - f_N^*\|_{L^2(\Omega)}^2 \leq \sum_{n\geq N} \theta_n(f) \leq \sum_{n\geq \frac{2 N}{3}} \theta_n^\omega(f) +  \sum_{n\geq \frac{N}{3}} \theta_n^\psi(f)
 =: T_1 + T_2,
\end{align}
where $(\theta_n^\omega(f))_{n\in\N}$ is the non-increasing rearrangement of $(|\left \langle f, \varphi_n\right \rangle_{L^2(\Omega)}|^2)_{\varphi_n \in \mathcal{W}_{t,\tau}}$
and $(\theta_n^\psi(f))_{n\in\N}$ the non-increasing rearrangement of $(|\left \langle f, \varphi_n\right \rangle_{L^2(\Omega)}|^2)_{\varphi_n \in \mathcal{S}_0}$, where
$\mathcal{S}_0 = \{ \psi_{j,k,m,\iota}\, : \, (j,k,m,\iota) \in \Lambda_0\}$.

We now estimate $T_1$ and $T_2$. By Theorem \ref{thm:compShearAreOptSparse},
\begin{equation} \label{eq:approx1}
T_2 \lesssim N^{-2}\log(N)^3.
\end{equation}
The sum $T_1$ corresponding to the wavelet part can be split into two parts again. First, we denote by $(\theta_n^\omega(f)^{(s)})_{n\in \N}$ those elements of
$(\theta_n^\omega(f))_{n\in \N}$ where $|\left \langle \varphi_n, f\right \rangle_{L^2(\Omega)}|^2  = \theta_n^\omega$ and $\suppp \varphi_n \cap \partial D
\neq \emptyset$. These are the wavelet elements corresponding to the smooth part of the function $f$. Second, we label the remaining elements by $\theta_n^\omega(f)^{(ns)}$ for $n \in \N$. Using similar arguments about the set of
largest coefficients as before, we obtain that
\begin{align*} 
T_1 \leq \sum_{n\geq \frac{N}{3}} \theta_n^\omega(f)^{(s)} + \sum_{n\geq \frac{N}{3}} \theta_n^\omega(f)^{(ns)}.
\end{align*}
By \eqref{eq:estForWavelets1},
\begin{equation*}
\sum_{n\geq \frac{N}{3}} \theta_n^\omega(f)^{(s)} \lesssim N^{-2} \quad \mbox{as } N\to \infty.
\end{equation*}
The wavelet coefficients corresponding to the non-smooth part of $f$ can be estimated by using that the boundary curve of $D$
intersects $\partial \Omega$ only finitely often. Therefore, due to the construction of the wavelet system $\mathcal{W}_{t,\tau}$, we obtain that only $\sim 2^{(1- \tau) j}< 2^{(2/3-\epsilon) j}$ wavelets intersect the boundary of $D$ where $0<\eps <\tau -1/3$. Furthermore, by the boundedness of $f$, we have  {$|\ \langle \omega_{j,m,\upsilon}, f \ \rangle_{L^2(\Omega)}|^2 \lesssim 2^{-2j}$}. Hence, we obtain
\begin{align*}
\sum_{n} (\theta_n^\omega(f)^{(ns)})^{\frac{1}{3}} \lesssim \sum_{j\in \N} 2^{(2/3-\epsilon) j} (2^{-2j})^{\frac{1}{3}} < \infty.
\end{align*}
Consequently, $(\theta_n^\omega(f)^{(ns)})_{n\in \N} \in \ell^{\frac{1}{3}}$ which, by the Stechkin Lemma (see for instance \cite{Dev1998}), yields
\begin{align} \label{eq:approx4}
\sum_{n \geq N} (\theta_n^\omega(f)^{(ns)}) \lesssim  N^{-2} \text{ for } N\to \infty.
\end{align}
Applying \eqref{eq:approx1}--\eqref{eq:approx4} to \eqref{eq:approx} proves the claim.
\end{proof}

\section{Numerical experiments}\label{sec:numerics}
We now numerically analyze some of the properties of boundary shearlet systems. Since estimates for frame bounds as derived in
Theorem \ref{thm:FrameProperty} are typically far from being tight, in Subsection \ref{sec:FrameProp} we numerically  {compute}
the frame bounds. In Subsections \ref{sec:LocGram} and \ref{sec:GelfPropNum} we then analyze the localization properties of the
Gramian and the Gelfand frame property, which are features of boundary shearlet systems the theoretical analysis of which was far beyond
the scope of this paper.

For all numerical experiments, we choose a digitized version $\mathbf{\Omega}$ of the domain $\Omega = [0,1]^2$ as an $n\times n$ pixel image. We will specify the number $n$ at the relevant points later. Our implementation of boundary
shearlet systems then uses the MATLAB toolboxes \texttt{WaveLab} from \url{http://statweb.stanford.edu/~wavelab/} and
\texttt{ShearLab} from \url{http://www.shearlab.org} for the implementation of the analysis and synthesis operator of boundary
shearlet systems. In \texttt{WaveLab} and \texttt{ShearLab}, the wavelet and shearlet elements are not normalized. Since this is crucial for the setting of bounded domains, we normalize all these functions. For later use, let $\bfTPw$ and $\bfTPs$
denote the implementation of the analysis operators of the wavelet and shearlet systems after normalization.

The definition of boundary shearlet systems requires a hybrid system consisting of a subset of the wavelet system and a subset
of the shearlet system. Concerning the wavelet elements, we only choose those which are close to the boundary. Depending on the
offset of the boundary shearlet system, we construct a mask, $\mathbf{M_w}$, for the wavelet system that restricts the analysis
operator to a subset of the full wavelet system. In the sequel, we will always choose $\tau= 1/3$ and only vary the offset $t$. Similarly, we need to subsample the shearlet system provided by \texttt{ShearLab}.
In fact, \texttt{ShearLab} provides a non-subsampled shearlet transform, i.e., it computes the shearlet coefficients using the
full system
$$
\left\{\psi_{j,k,(S_k A_j m),\iota} \, : \, j\leq J, \  \iota \in \{-1,0,1\}, \ |k|\leq |\iota |2^{\lfloor j/2 \rfloor}, \ m \in c \Z^2, \right\}.
$$ On the other hand the theory requires us to restrict to the shearlet system
$$
\left\{ \psi_{j,k,m,\iota} \, : \, j\leq J, \  \iota \in \{-1,0,1\}, \ |k|\leq |\iota |2^{\lfloor j/2 \rfloor}, \ m \in c \Z^2, \right\}.
$$
Furthermore, we exclude shearlets from our system that intersect the boundary of $\Omega$. We incorporate all of these requirements
in a mask $\mathbf{M_s}$.

The analysis operator of the combined system is now given by
\begin{equation} \label{eq:analysisop}
\bfTP: = \left(\begin{array}{c}
 \mathbf{M_w }\bfTPw \\
\mathbf{M_s} \bfTPs
\end{array}\right).
\end{equation}
Using these operators, we derive an implementation of the synthesis operator of boundary shearlet systems by using
\[
\bfTPA = \mathbf{M_w T_{\Phi_w}^*} +  \mathbf{M_s T_{\Phi_s}^*}.
\]
The implementation of the frame operator is given by
\[
\mathbf{S} :=  \mathbf{T_{\Phi_w}^*M_w T_{\Phi_w}} + \mathbf{T_{\Phi_s}^*M_s T_{\Phi_s}}.
\]
To apply the inverse frame operator $\mathbf{S^{-1}}$, we use MATLAB's build-in \emph{conjugate gradients method}, \texttt{pcg}.

\newcommand{\FrameQuotientsN}{512}
\newcommand{\FrameQuotientsNrScales}{5}

\subsection{Frame properties}\label{sec:FrameProp}

We now compute the frame bounds of a boundary shearlet system for various offsets of the wavelet part. We pick an $\FrameQuotientsN
\times \FrameQuotientsN$ pixel  {domain} as a digitization of $\Omega$. The wavelet and shearlet systems are computed using $\FrameQuotientsNrScales$
scales. Since the optimal frame bounds $A$ and $B$ are the extremal points of the spectrum of the frame operator of the system, we
numerically compute them for this boundary shearlet system by computing the smallest and largest eigenvalues of $\mathbf{S}$. {For this task we used MATLAB's built-in method \texttt{eigs}.} In
Figure \ref{fig:FrameQuotient}, we depict the quotient $B/A$ for varying offset.

\begin{figure}[htb]
\centering
\includegraphics[width = 0.6\textwidth]{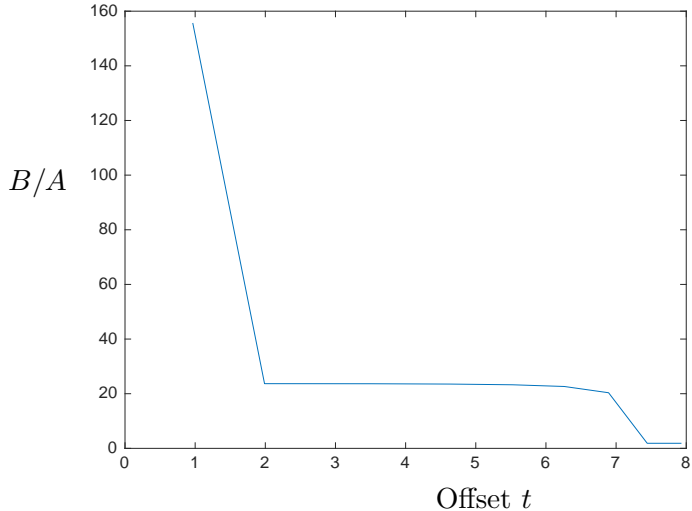}
\put(-120,0){Offset $t$}
\put(-280,120){$B/A$}
\caption{Quotient of the frame bounds for varying offset. One observes that for high offset the quotient becomes stable and explodes for decreasing offset.}
\label{fig:FrameQuotient}
\end{figure}

We observe that for larger offset the ratio of the frame bounds is somehow not too far from $1$, which provides us with reasonably
good condition numbers for the computation of $\mathbf{S}^{-1}$. In fact, the values of these quotients are comparable with those
of the full shearlet system used in \texttt{ShearLab} \cite{shearLab}. As expected, the frame property breaks down, when the
offset becomes too small. This is in accordance with Theorem \ref{thm:FrameProperty}.

\subsection{Localization of the Gramian}\label{sec:LocGram}

\newcommand{\GramianN}{256}
\newcommand{\GramianNrScales}{4}
\newcommand{\GramianNrDirections}{[1 \ 1 \ 2 \ 2]}

Using the analysis operator as defined in \eqref{eq:analysisop}, the Gramian of the boundary shearlet system is given by
\begin{align*}
\mathbf{G} := \bfTP \bfTPA.
\end{align*}
The linear operators $\bfTP$ and $\bfTPA$ are implemented using the \texttt{Spot} Toolbox, which is available at
\url{http://www.cs.ubc.ca/labs/scl/spot/index.html}. The matrix representation of the Gramian is shown in Figure
\ref{fig:GramianCombined}. It is clearly visible, that the Gramian of the boundary shearlet system has diagonal
structure. 

The figures were produced for a $\GramianN \times \GramianN$ digitization of $\Omega$, $\GramianNrScales$ scales
in the boundary shearlet system with the number of directions being $\GramianNrDirections$.

\begin{figure}[htb!]
\centering
\includegraphics[width = 0.95 \textwidth]{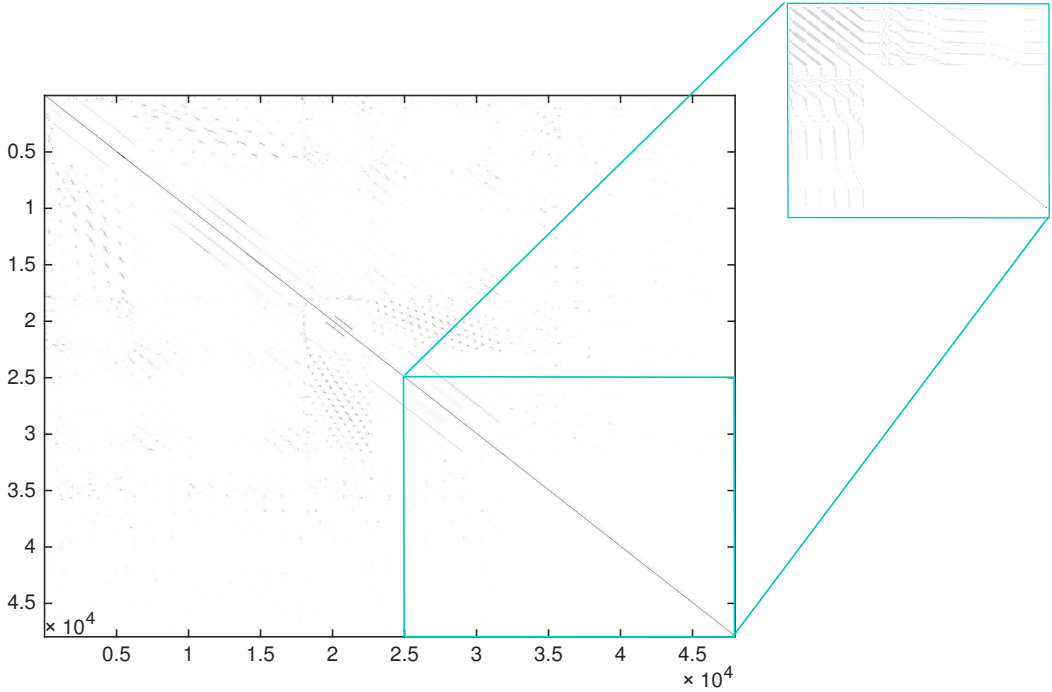}
\caption{Gramian of the boundary shearlet system. The part zoomed region is shown with changed contrast for better visualization of the different sparsity patterns of shear-shear, shear-wave and wave-wave. \label{fig:GramianCombined}}
\end{figure}


%

\subsection{Numerical analysis of the Gelfand frame property}\label{sec:GelfPropNum}

\newcommand{\GelfandN}{1024}
\newcommand{\GelfandNrScales}{6}
\newcommand{\GelfandNrDirections}{[1 \ 1 \ 2 \ 2 \ 3]}

We have already discussed the Gelfand frame property in Subsection \ref{sec:GelfProp}. Indeed we have shown that the lower inequality, i.e., (GFA1) can be achieved. However, as we have also explained in the same subsection, our theoretical analysis is restricted to (GFA1) only since (GFA2) involves the dual frame which is not available. Nevertheless, property (GFA2) will be confirmed numerically in this subsection.

We will now proceed as follows, first, we require a numerically computable discretization of property (GFA2). For this, notice that employing the characterization
of $H^s(\Omega)$ by a boundary wavelet frame and an appropriate weight in the sense
that
$$
\|T_{\Phi_w} c\|^2_{H^s(\Omega)} \sim \|c\|^2_{\ell^{2,s}},
$$
one can obtain the following property that is equivalent to (GFA2):
\begin{align} \label{eq:contUpBound}
\|  \langle T_{\Phi_w} c, \varphi_n^d \rangle_{L^2(\Omega)}\|_{\ell^{2,s}}^2 \lesssim \|c\|_{\ell^{2,s}}^2 \fa c \in \ell^{2,s}.
\end{align}
To derive a discrete analogue of \eqref{eq:contUpBound}, we first let
\[
W: \ell^{2,s} \to \ell^2, (x_k)_k \mapsto (w_k\cdot x_k)_k
\]
be the canonicial isometry. Furthermore, since $\varphi_n^d = S^{-1} \varphi_n$, it follows that
\[
\langle T_{\Phi_w} c, \varphi_n^d \rangle_{L^2(\Omega)} = \langle S^{-1} T_{\Phi_w} c, \varphi_n \rangle_{L^2(\Omega)}.
\]
Using the canonical discretization of $W$ as a diagonal matrix, we obtain two matrices $\mathbf{W}$ and $\mathbf{W_w}$ adapted to the
indexing of the boundary shearlet system and the full wavelet system, respectively. The discrete analogue of \eqref{eq:contUpBound} now
takes the form
\begin{align*}
\|\mathbf{W} \bfTP(\mathbf{S^{-1}}\mathbf{T}_{\Phi_w} c)\|^2 \lesssim \|\mathbf{W_w} c\|^2 \fa c \in  \R^{n^2}.
\end{align*}
In order to examine this bound and check its validity for our boundary shearlet system in the discrete setting, we estimate
\begin{align*}
\max_{\|c\| = 1} \|\mathbf{W} \bfTP\mathbf{S^{-1}}\mathbf{T}_{\Phi_w} \mathbf{W_w}^{-1} c\|^2
\end{align*}
by computing the square-root of the largest eigenvalue of
\begin{align}\label{eq:theOperator}
\mathbf{W_w}^{-1} \mathbf{T}_{\Phi_w} \mathbf{S^{-1}} \bfTPA \mathbf{W}^2 \bfTP\mathbf{S^{-1}}\mathbf{T}_{\Phi_w} \mathbf{W_w}^{-1}.
\end{align}

\changed{
\begin{table}[htb]
\centering
\begin{tabular}{| l | c | c | c | c | c|} \hline
  Offset & $s$ = 0 &  $s$ = 0.5 & $s$= 1 & $s$= 1.5 \\
  \hline
    7.31   & 3.77  &  3.77  &  3.77 &   3.77\\
    6.59   & 3.78  &  3.78  &  3.78 &   3.78\\
    5.72   & 3.78  &  3.78  &  3.78 &   3.78\\
    4.63   & 3.79  &  3.78  &  3.78 &   3.78\\
    3.18   & 3.79  &  3.79  &  4.09 &   6.11\\
    0.97   & 3.79  &  3.79  &  4.15 &   6.81\\
    0.35   & 3.79  &  3.79  &  4.35 &  13.10\\
\hline
\end{tabular}
\caption{Square root of largest singular values of $\mathbf{W} \bfTP\mathbf{S^{-1}}\mathbf{T}_{\Phi_w} \mathbf{W_w}^{-1}$ for varying offset and Sobolev parameter $s$.}\label{Table1}
\end{table}}

In Table \ref{Table1}, we depict the square-root of the largest eigenvalue of the operator \eqref{eq:theOperator}
with different weights $\mathbf{W}, \mathbf{W_w}$ and different offset for $n = 1024$. 
\ref{Table1}. In this numerical experiment, the weights are chosen as $2^{j s}$, where $j$ describes the scale of the frame element,
both of wavelet and shearlet, and $s$ is a parameter that takes values between $0$ and $1.5$. The shearlet and wavelet systems were
constructed with $\GelfandNrScales$ scales.


In Table \ref{Table1}, one can observe that, although the largest eigenvalues
of {\eqref{eq:theOperator}} increase with growing Sobolev parameter, they do so
remarkably slow if the offset is sufficiently high. Thus we conclude that our experiments demonstrate the proper mapping properties of the dual frame.

\section*{Acknowledgements}

P. Petersen would like to thank Kristof Schr\"oder and Massimo Fornasier for various discussions on related topics. Furthermore,
P. Petersen expresses his gratitude to Massimo Fornasier and the Technische Universit\"at M\"unchen for the hospitality during
P. Petersen's research visit. Parts of this work was also done when J. Ma visited the Department of Applied Mathematics and
Theoretical Physics of the University of Cambridge, and he is grateful for its hospitality. J. Ma, P. Petersen and M. Raslan acknowledge
support from the DFG Collaborative Research Center TRR 109 ``Discretization in Geometry and Dynamics''; they are also supported
by the Berlin Mathematical School. 

G. Kutyniok acknowledges support by the Einstein Foundation Berlin, by the Einstein Center
for Mathematics Berlin (ECMath), by Deutsche Forschungsgemeinschaft (DFG) SPP 1798, by the DFG Collaborative Research Center TRR
109 ``Discretization in Geometry and Dynamics'', by the DFG Research Center {\sc Matheon} ``Mathematics for key technologies''
in Berlin, and by the European Commission-Project DEDALE (contract no.665044) within the H2020 Framework Program. Parts of this work was accomplished while G. Kutyniok was visiting ETH Z\"urich. She is grateful to the Institute for
Mathematical Research (FIM) and the Seminar for Applied Mathematics (SAM) for their hospitality and support during this visit.
\small
\bibliographystyle{abbrv}
\bibliography{Ref_Shearlets_On_Bounded_Domains}

\appendix
\section{$H^s(\R^2)$ Frame property of a reweighted shearlet system} \label{sec:ProofHsR2}

This section shall be concerned with the proof of Theorem \ref{Thm:HsR2Frame}. The proof is split into Lemma \ref{lemm:upperBdSobolev} for the upper frame bound and Lemma \ref{lemm:lowerBdSobolev} for the lower frame bound. 

We use the following lemma to estimate weighted $L^2(\R^2)$ - frame coefficients by the $H^s(\R^2)$ norm of a function $f$.
\begin{lemma}[\cite{DissPP}, Lemma 2.4.9.] \label{lemm:0}
Let $s \in \N$, $\tilde{C}>0,$ $\beta_{\mathrm{sh}}>3$ and $\alpha_{\mathrm{sh}} > \beta_{\mathrm{sh}} + s$. Further, let $\phi, \psi, \psitilde \in L^2(\R^2)$ such that for almost all $\xi \in \R^2$ there holds  
\begin{align*}
|\hat \phi(\xi)| &\leq \tilde{C} \min\{1, |\xi_1 |^{-\beta_{\mathrm{sh}}}\}\min\{1, |\xi_2 |^{-\beta_{\mathrm{sh}}}\}\nonumber\\
|\hat \psi(\xi)| &\leq \tilde{C} \min\{1, |\xi_1 |^{\alpha_{\mathrm{sh}}}\}\min\{1, |\xi_1 |^{-\beta_{\mathrm{sh}}}\}\min\{1, |\xi_2 |^{-\beta_{\mathrm{sh}}}\}  \\
|\hat {\tilde{\psi}}(\xi)| &\leq \tilde{C} \min\{1, | \xi_2 |^{\alpha_{\mathrm{sh}}}\}\min\{1, |\xi_1 |^{-\beta_{\mathrm{sh}}}\}\min\{1, |\xi_2 |^{-\beta_{\mathrm{sh}}}\}.
\nonumber
\end{align*}
Then there exists a constant $C>0$ depending only on $s, \alpha_{\mathrm{sh}}$ and $\beta_{\mathrm{sh}}$ such that $\mathcal{SH}(\phi, \psi, \tilde{\psi}, c)$ satisfies
$$
\left\|\left(2^{js}\langle f, \psi_{j,k,m,\iota}\rangle_{L^2(\R^2)}\right)_{(j,k,m,\iota)\in \Lambda} \right\|_{\ell^2(\Lambda)}\leq \frac{C \tilde{C} }{\sqrt{|\det(M_c)|}} \|f\|_{H^s(\R^2)} \fa f \in H^s(\R^2).
$$
\end{lemma}

We continue to give the upper frame bound for the weighted shearlet system in $H^s(\R^2)$.

\begin{lemma}\label{lemm:upperBdSobolev}
For $s\in\N$ let $\phi^1, \psi^1 \in L^2(\R)$ be such that for some $\tilde{C} \geq 0$, $\alpha_{\mathrm{sh}} - s> \beta_{\mathrm{sh}} > 3$, and all $0\leq r_1,r_2 \leq 2s$, 
\begin{align}\label{eq:decayAssumptions}
    \widehat{D^{r_1}\phi^1}(\xi) &\leq \sqrt{\tilde{C}} \min \{1,|\xi|^{-\beta_{\mathrm{sh}}} \}, \text{ for all } \xi \in \R, \nonumber\\
    \widehat{D^{r_2}\psi^1}(\xi) &\leq \sqrt{\tilde{C}} \min \{1,|\xi|^{\alpha_{\mathrm{sh}}} \} \min \{1,|\xi|^{-\beta_{\mathrm{sh}}} \}, \text{ for all } \xi \in \R.
\end{align}
We further denote $\phi = \phi^1 \otimes \phi^1$ and $\psi = \psi^1 \otimes \phi^1$ and $\tilde{\psi}(x_1, x_2) := \psi(x_2, x_1)$ for all $x\in \R^2$.
Then, there exists some $B>0$ depending only on $\alpha_{\mathrm{sh}}, \beta_{\mathrm{sh}}$ and $s$ such for $(\psi_{j,k,m,\iota})_{(j,k,m,\iota)\in \Lambda} = \mathcal{SH}(\phi, \psi, \tilde{\psi}, (c_1,c_2))$ we have the estimate
\begin{align}\label{eq:theUpperHsEstimate}
\sum_{(j,k,m,\iota)\in \Lambda} |\langle f, 2^{-js} \psi_{j,k,m,\iota} \rangle_{H^s(\R^2)}  |^2 \leq B \frac{\tilde{C}}{|\det(M_c)|} \|f\|_{H^s(\mathbb{R}^2)}^2
\end{align}
for all $f \in H^s(\R^2)$.
\end{lemma}
\begin{proof}
We define for $\mathbf{r}= (r_1,r_2)$: $\phi^{(\mathbf{r})} = D^{r_1}\phi^1 \otimes D^{r_2}\phi^1$, $\psi^{(\mathbf{r})} = D^{r_1}\psi^1\otimes D^{r_2}\phi^1$, and $\tilde{\psi}^{(\mathbf{r})}(x_1,x_2)=\psi^{(\mathbf{r})}(x_2,x_1)$.
By Lemma \ref{lemm:0} there exists a constant $C$, dependent only on $\alpha_{\mathrm{sh}}, \beta_{\mathrm{sh}},$ such that for all $0\leq r_1,r_2 \leq 2s$ and for $(\psi_{j,k,m,\iota}^{(\mathbf{r})})_{(j,k,m,\iota)\in \Lambda} = \mathcal{SH}(\phi^{(\mathbf{r})}, \psi^{(\mathbf{r})}, \tilde{\psi}^{(\mathbf{r})}, c)$ we have that
\begin{align}\label{eq:TheLtwoEstimate}
    \sum_{(j,k,m,\iota)\in \Lambda} 2^{2js}|\langle f, \psi_{j,k,m,\iota}^{(\mathbf{r})}\rangle_{L^2(\R^2)} |^2  \leq C \frac{\tilde{C}}{|\det(M_c)|} \|f\|_{H^s(\R^2)}^2, \text{ for all } f\in H^s(\R^2).
\end{align}
From the assumptions \eqref{eq:decayAssumptions} we have that $\phi, \psi, \tilde{\psi} \in H^{2s}(\R^2)$ and thus we can estimate the left hand side of \eqref{eq:theUpperHsEstimate}. 
\begin{align}
    &\sum_{(j,k,m,\iota) \in \Lambda} | 2^{-js}\left \langle f, \psi_{j,k,m,\iota}\right \rangle_{H^s(\R^2)} |^2 \nonumber\\
    & \qquad = \sum_{(j,k,m,\iota) \in \Lambda} | \sum_{|\mathbf{a}|\leq s} 2^{-js} \left \langle D^{\mathbf{a}}f, D^{\mathbf{a}}\psi_{j,k,m,\iota}\right \rangle_{L^2(\R^2)} |^2\nonumber\\
    & \qquad  \lesssim  \sum_{(j,k,m,\iota) \in \Lambda}  \sum_{|\mathbf{a}|\leq s}|2^{-js} \left \langle f, D^{2\mathbf{a}}\psi_{j,k,m,\iota}\right \rangle_{L^2(\R^2)} |^2 = \mathrm{I}. \label{eq:weNeedToContinueHereAfterAnApplicationOfTheChainRule}
\end{align}
By using the product rule we calculate 
\begin{align}
   D^{2\mathbf{a}}\psi_{j,k,m,\iota}  = \sum_{0\leq r_1,r_2 \leq 2s} c_{\mathbf{r}}^{(\mathbf{a})} \psi^{(\mathbf{r})}_{j,k,m,\iota},\label{eq:chainRuleTrick}
\end{align}
with coefficients $c_{\mathbf{r}}^{\mathbf{a}}$ bounded in absolute value by $2^{2 j s}$.
After applying \eqref{eq:chainRuleTrick} to \eqref{eq:weNeedToContinueHereAfterAnApplicationOfTheChainRule} we estimate 
\begin{align*}
    \mathrm{I} &\lesssim \sum_{(j,k,m,\iota) \in \Lambda}  \sum_{|\mathbf{a}|\leq s} \sum_{0\leq r_1, r_2 \leq 2s} 2^{2js} | \langle f, \psi_{j,k,m,\iota}^{(\mathbf{r})}\rangle_{L^2(\R^2)}|^2\\
    &\lesssim \sup_{0\leq r_1,r_2 \leq 2s}\sum_{(j,k,m,\iota) \in \Lambda}  2^{2js} | \langle f, \psi_{j,k,m,\iota}^{(\mathbf{r})} \rangle_{L^2(\R^2)} |^2 =\mathrm{II}.
\end{align*}
Invoking \eqref{eq:TheLtwoEstimate}, there exists some $B>0$ such that
\begin{align*}
    \mathrm{II} \leq B \frac{\tilde{C}}{|\det(M_c)|} \|f\|_{H^s(\R^2)}^2, \text{ for all } f\in H^s(\R^2).
\end{align*}
This yields the result.
\end{proof}
As a next step we provide the lower frame bound.

\begin{lemma}\label{lemm:lowerBdSobolev}
For $s\in\N$, $\alpha_{\mathrm{sh}} - s> \beta_{\mathrm{sh}} > 4$ and some $\tilde{C}>0$ let $\phi^1, \psi^1 \in L^2(\R)$ be such that for all $0\leq r_1,r_2 \leq 2s$,
\begin{align*}
    \widehat{D^{r_1}\phi^1}(\xi) &\leq  \sqrt{\tilde{C}}\min \{1,|\xi|^{-\beta_{\mathrm{sh}}} \}, \text{ for all } \xi \in \R, \nonumber\\
    \widehat{D^{r_2}\psi^1}(\xi) &\leq \sqrt{\tilde{C}}  \min \{1,|\xi|^{\alpha_{\mathrm{sh}}} \} \min \{1,|\xi|^{-\beta_{\mathrm{sh}}} \}, \text{ for all } \xi \in \R.
\end{align*}
Further, let $\phi, \psi, \psitilde, \theta, \thetatilde, \mu$ be as in \eqref{eq:thenotation}. Assume that there exists $\bar c>0$ such that for all $c_1,c_2\leq \bar{c}$ we have that $\mathcal{SH}(\mu,\theta,\tilde{\theta},(c_1,c_2))$ forms a frame for $L^2( \mathbb R^2)$ with lower frame bound, which can be bounded from below independently of $c_1,c_2$. 
Then there exists $\tilde{c}>0$ such that for all $c_1=c_2 \leq \tilde{c}$ and $c = (c_1,c_2)$ the system $(\psi_{j,k,m,\iota})_{(j,k,m,\iota)\in \Lambda} = \mathcal{SH}(\mu,\theta,\tilde{\theta},c)$, obeys
$$
\|f\|_{H^s(\mathbb{R}^2)}^2\lesssim  \sum_{(j,k,m,\iota)\in \Lambda} |\langle f, 2^{-js} \psi_{j,k,m,\iota} \rangle_{H^s(\R^2)}  |^2,
$$
for all $f\in H^s(\R^2)$.
\end{lemma}
\begin{proof}
Let $\bar{c} \geq c>0$. We define
\[ 
\hat{\mu}^c:=\mathcal{X}_{[-\frac{1}{2c_1},\frac{1}{2c_1}]^2}\hat{\mu}, \qquad \hat{\theta}^c:=\mathcal{X}_{[-\frac{1}{2c_1},\frac{1}{2c_1}]^2}\hat{\theta} \qquad, \hat{\thetatilde}^c:=\mathcal{X}_{[-\frac{1}{2c_1},\frac{1}{2c_1}]^2}\hat{\thetatilde}.
\]
By using the fact, that $\mathcal{SH}(\mu,\theta,\tilde{\theta},c)$ is a frame for $L^2( \R^2),$ we deduce
\begin{align}
    \|f\|_{H^s(\mathbb{R}^2)}^2\ &\sim \| (1+|\cdot|^2)^{\frac{s}{2}} \hat{f}\|_{L^2(\R^2)}^2\nonumber\\
    & \lesssim \sum_{(j,k,m,\iota)\in\Lambda} |\langle (1+|\cdot|^2)^{\frac{s}{2}} \hat{f}, \widehat{\theta}_{j,k,m,\iota} \rangle_{L^2(\R^2)}  |^2 \nonumber \\
    &\leq \sum_{(j,k,m,\iota)\in\Lambda} |\langle (1+|\cdot|^2)^{\frac{s}{2}} \hat{f}, \widehat{\theta}^c_{j,k,m} \rangle_{L^2(\R^2)}|^2 \nonumber \\
    &\qquad+ \sum_{(j,k,m,\iota)\in\Lambda} |\langle (1+|\cdot|^2)^{\frac{s}{2}} \hat{f}, \widehat{\theta}_{j,k,m}-\widehat{\theta}^c_{j,k,m} \rangle_{L^2(\R^2)}|^2 = \mathrm{I}_1 + \mathrm{I}_2. \label{eq:TermI2} 
\end{align}
Since $\beta_{\mathrm{sh}} > 4$, there exists $\epsilon >0$ such that $\beta_{\mathrm{sh}} - 1-\epsilon >3$ and for all $\xi \in \R^2$
\begin{align*}
    \min \{1, |\xi|^{\alpha_{\mathrm{sh}}}\} \min \{ 1, |\xi|^{-\beta_{\mathrm{sh}}}\}
    \leq \min \{1, |\xi_1|^{\alpha_{\mathrm{sh}}}\} \min \{ 1, |\xi_1|^{-\beta_{\mathrm{sh}}+1+\epsilon}\} (\max \{ 1, |\xi_1|\})^{-1-\epsilon}.
\end{align*}
Thus, we have that $(\theta_{j,k,m,\iota}- \theta_{j,k,m,\iota}^c)_{(j,k,m,\iota) \in \Lambda}$ satisfies the assumptions of Lemma \ref{lemm:upperBdSobolev} with $\sqrt{\tilde{C}} = 2c_1^{1+\epsilon}$. Hence, $\mathrm{I}_2$ can be estimated by $c_1^{2+2\epsilon}/|\det{(M_c)}| \|f\|_{H^s(\R^2)}^2$ and since $\det{(M_c}) = c_1^2$ this term is negligible for $c_1$ small enough.
Therefore we obtain
\begin{align*}
\|f\|_{H^s(\R^2)}^2 & \lesssim \sum_{(j,k,m,\iota)\in\Lambda} |\langle (1+|\cdot|^2)^{\frac{s}{2}} \hat{f}, \widehat{\theta}^c_{j,k,m} \rangle_{L^2(\R^2)}|^2\\
&= \sum_{(j,k,m)\in\Lambda^{'}} |\langle (1+|\cdot|^2)^{\frac{s}{2}} \hat{f}, \widehat{\theta}^c_{j,k,m} \rangle_{L^2(\R^2)}  |^2\\
    &\qquad+ \sum_{(j,k,m)\in\Lambda^{'}} |\langle (1+|\cdot|^2)^{\frac{s}{2}} \hat{f},\widehat{\thetatilde}^c_{j,k,m} \rangle_{L^2(\R^2)}  |^2 \\
    &\qquad+ \sum_{m\in\Z^2} |\langle (1+|\cdot|^2)^{\frac{s}{2}} \hat{f}, \mathcal{F}(\mu^c(\cdot-c_1m)) \rangle_{L^2(\R^2)}  |^2 \\
&= \sum_{(j,k,m)\in\Lambda^{'}} 2^{\frac{3j}{2}}\bigg\vert \int_{\R^2} \hat{f}(\xi)(1+|\xi|^2)^{\frac{s}{2}} \overline{\Fcal(\theta^c(S_kA_j\cdot-M_cm))(\xi)}d\xi \bigg\vert^2 \\
&\qquad +\sum_{(j,k,m)\in\Lambda^{'}} 2^{\frac{3j}{2}}\bigg\vert \int_{\R^2} \hat{f}(\xi)(1+|\xi|^2)^{\frac{s}{2}} \overline{\Fcal(\thetatilde^c(S_k^T\tilde{A}_j\cdot-M_{\tilde{c}}m))(\xi)}d\xi \bigg\vert^2 \\
&\qquad +\sum_{m\in\Z^2} \bigg\vert\int_{\R^2}\hat{f}(\xi)(1+|\xi|^2)^{\frac{s}{2}}\overline{\Fcal(\mu^c(\cdot-c_1m))(\xi)} d\xi\bigg\vert^2 \\
&= \sum_{(j,k,m)\in\Lambda^{'}} 2^{-\frac{3j}{2}}\bigg\vert \int_{\R^2} \hat{f}(\xi)(1+|\xi|^2)^{\frac{s}{2}} \overline{\hat{\theta}^c(S_k^{-T}A_j^{-1}\xi)}e^{2\pi i\langle M_cS_k^{-T}A_j^{-1}\xi, m \rangle}d\xi \bigg\vert^2 \\
&\qquad +\sum_{(j,k,m)\in\Lambda^{'}} 2^{-\frac{3j}{2}}\bigg\vert \int_{\R^2} \hat{f}(\xi)(1+|\xi|^2)^{\frac{s}{2}} \overline{\hat{\thetatilde}^c(S_k^{-1}\tilde{A}_j^{-1}\xi)} e^{2\pi i\langle M_{\tilde{c}}S_k^{-1}\tilde{A}_{j}^{-1}\xi,m\rangle}d\xi \bigg\vert^2 \\
&\qquad +\sum_{m\in\Z^2} \bigg\vert\int_{\R^2}\hat{f}(\xi)(1+|\xi|^2)^{\frac{s}{2}}\overline{\hat{\mu}^c(\xi)}e^{2\pi i\langle c_1\xi,m \rangle} d\xi\bigg\vert^2 = \mathrm{II}.
\end{align*}

Now we substitute 
\begin{align}\label{eq:transformI}
\xi\leadsto A_jS_k^TM_c^{-1}\xi, \qquad \xi\leadsto \tilde{A}_jS_kM_{\tilde{c}}^{-1}\xi, \qquad \xi\leadsto \frac{\xi}{c_1}.    
\end{align} 
Furthermore, we use that $\mu^c, \theta^c,$ and $\tilde{\theta}^c$ are all supported in $[-\frac{1}{2c_1},\frac{1}{2c_1}]^2.$

\begin{align*}
\mathrm{II}&= \sum_{(j,k,m)\in\Lambda^{'}} \frac{2^{\frac{3j}{2}}}{|\text{det}(M_c)|^2}\cdot \bigg\vert \int_{[-\frac{1}{2},\frac{1}{2}]^2} \hat{f}(A_jS_k^TM_c^{-1}\xi)(1+|A_jS_k^TM_c^{-1}\xi|^2)^{\frac{s}{2}}\overline{\hat{\theta}^c(M_c^{-1}\xi)}e^{2\pi i\langle \xi,m \rangle} d\xi\bigg\vert^{2}  \\
&\qquad+ \sum_{(j,k,m)\in\Lambda^{'}} \frac{2^{\frac{3j}{2}}}{|\text{det}(M_c)|^2}\cdot \bigg\vert \int_{[-\frac{1}{2},\frac{1}{2}]^2} \hat{f}(\tilde{A}_jS_kM_{\tilde{c}}^{-1}\xi)(1+|\tilde{A}_jS_k M_{\tilde{c}}^{-1}\xi|^2)^{\frac{s}{2}}\overline{\hat{\thetatilde}^c(M_{\tilde{c}}^{-1}\xi)}e^{2\pi i\langle \xi,m \rangle} d\xi\bigg\vert^{2}  \\
&\qquad+ \sum_{m\in\Z^2} \frac{1}{c_1^2}\bigg\vert \int_{[-\frac{1}{2},\frac{1}{2}]^2} \hat{f}\left(\frac{\xi}{c_1} \right) \left( 1+\bigg\vert \frac{\xi}{c_1} \bigg\vert^2 \right)^{\frac{s}{2}} \overline{\hat{\mu}^c\left(\frac{\xi}{c_1}\right)}e^{2\pi i\langle\xi,m\rangle}d\xi\bigg\vert^2.
\end{align*}
By using the Parseval identity we obtain
\begin{align*}
\mathrm{II}&= \sum_{(j,k)\in\Lambda^{''}} \frac{2^{\frac{3j}{2}}}{|\text{det}(M_c)|^2}\cdot \left\|  \hat{f}(A_jS_k^TM_c^{-1}\cdot)(1+|A_jS_k^TM_c^{-1}\cdot|^2)^{\frac{s}{2}}\hat{\theta}^c(M_c^{-1}\cdot)\right\|_{L^2(\R^2)}^{2}  \\
&\qquad+ \sum_{(j,k)\in\Lambda^{''}} \frac{2^{\frac{3j}{2}}}{|\text{det}(M_c)|^2}\cdot \left\| \hat{f}(\tilde{A}_jS_kM_{\tilde{c}}^{-1}\cdot)(1+|\tilde{A}_jS_k M_{\tilde{c}}^{-1}\cdot|^2)^{\frac{s}{2}}\hat{\thetatilde}^c(M_{\tilde{c}}^{-1}\cdot)\right\|_{L^2(\R^2)}^{2}  \\
&\qquad+ \frac{1}{c_1^2}\left\| \hat{f}\left(\frac{\cdot}{c_1} \right) \left( 1+\bigg\vert \frac{\cdot}{c_1} \bigg\vert^2 \right)^{\frac{s}{2}} \hat{\mu}^c\left(\frac{\cdot}{c_1}\right)\right\|_{L^2(\R^2)}^2.    
\end{align*}
Now we substitute 
\begin{align}\label{eq:transformII}
\xi\leadsto M_c\xi, \qquad \xi\leadsto M_{\tilde{c}}\xi, \qquad \xi\leadsto c_1\xi
\end{align}
to arrive at
\begin{align*}
\mathrm{II}&= \sum_{(j,k)\in\Lambda^{''}} 2^{\frac{3j}{2}} \left\| (1+|A_jS_k^T\cdot|^2)^{\frac{s}{2}}\hat{f}(A_jS_k^T\cdot) \hat{\theta}^c(\cdot) \right\|_{L^2(\R^2)}^2  \\
&\quad+  \sum_{(j,k)\in\Lambda^{''}}2^{\frac{3j}{2}} \left\| (1+|\tilde{A}_jS_k\cdot|^2)^{\frac{s}{2}} \hat{f}(\tilde{A}_jS_k\cdot) \hat{\thetatilde}^c(\cdot) \right\|_{L^2(\R^2)}^2  \\
&\qquad + \left\| (1+|\cdot|^2)^\frac{s}{2}\cdot\hat{f}\cdot \hat{\phi} \right\|_{L^2(\R^2)}^2.
\end{align*}

Now set 
\[ 
\hat{\phi}^c:=\mathcal{X}_{[-\frac{1}{2c_1},\frac{1}{2c_1}]^2}\hat{\phi}, \qquad \hat{\psi}^c:=\mathcal{X}_{[-\frac{1}{2c_1},\frac{1}{2c_1}]^2}\hat{\psi} \qquad, \hat{\psitilde}^c:=\mathcal{X}_{[-\frac{1}{2c_1},\frac{1}{2c_1}]^2}\hat{\psitilde}.
\]
Using the definition of $\theta^c,\thetatilde^c,\mu^c$ then yields

\begin{align*}
\mathrm{II}&=\sum_{(j,k)\in\Lambda^{''}} 2^{\frac{3j}{2}} \left\| (1+|A_jS_k^T\cdot|^2)^{\frac{s}{2}}\hat{f}(A_jS_k^T\cdot) |(\cdot)_1|^s \hat{\psi}^c(\cdot) \right\|_{L^2(\R^2)}^2  \\
&\quad+  \sum_{(j,k)\in\Lambda^{''}}2^{\frac{3j}{2}} \left\| (1+|\tilde{A}_jS_k\cdot|^2)^{\frac{s}{2}} \hat{f}(\tilde{A}_jS_k\cdot) |(\cdot)_2|^s \hat{\psitilde}^c(\cdot) \right\|_{L^2(\R^2)}^2  \\
&\qquad + \left\| (1+|\cdot|^2)^\frac{s}{2}\cdot\hat{f}\cdot \hat{\phi} \right\|_{L^2(\R^2)}^2 \\    
&= \sum_{(j,k)\in\Lambda^{''}} 2^{\frac{3j}{2}} \left\| (1+|A_jS_k^T\cdot|^2)^{\frac{s}{2}}\hat{f}(A_jS_k^T\cdot) 2^{-js} |(A_jS_k^T\cdot)_1|^s \hat{\psi}^c(\cdot) \right\|_{L^2(\R^2)}^2  \\
&\quad+  \sum_{(j,k)\in\Lambda^{''}}2^{\frac{3j}{2}} \left\| (1+|\tilde{A}_jS_k\cdot|^2)^{\frac{s}{2}} \hat{f}(\tilde{A}_jS_k\cdot) 2^{-js} |(\tilde{A}_jS_k\cdot)_2|^s \hat{\psi}^c(\cdot) \right\|_{L^2(\R^2)}^2  \\
&\qquad + \left\| (1+|\cdot|^2)^\frac{s}{2}\cdot\hat{f}\cdot \hat{\phi}^c \right\|_{L^2(\R^2)}^2 \\
&\leq \sum_{(j,k)\in\Lambda^{''}} 2^{\frac{3j}{2}}\left\| (1+|A_jS_k^T\cdot|^2)^{s}\hat{f}(A_jS_k^T\cdot) 2^{-js}  \hat{\psi}^c(\cdot) \right\|_{L^2(\R^2)}^2  \\
&\quad+  \sum_{(j,k)\in\Lambda^{''}} 2^{\frac{3j}{2}}\left\| (1+|\tilde{A}_jS_k\cdot|^2)^{s}  \hat{f}(\tilde{A}_jS_k\cdot) 2^{-js}  \hat{\psitilde}^c(\cdot) \right\|_{L^2(\R^2)}^2  \\
&\qquad + \left\| (1+|\cdot|^2)^s\cdot\hat{f}\cdot \hat{\phi}^c \right\|_{L^2(\R^2)}^2  = \mathrm{III}.
\end{align*}
We observe that
$$
(1+|\xi|^2)^s \lesssim \sum_{|\mathbf{a}|\leq s} (2\pi)^{2|\mathbf{a}|} \xi^{2\mathbf{a}} \text{ for all } \xi \in \R^2.
$$
Thus we can estimate 
\begin{align*}
    \mathrm{III} \lesssim &\sum_{(j,k)\in\Lambda^{''}}\nonumber 2^{\frac{3j}{2}}\cdot\left\|\sum_{|\mathbf{a}|\leq s} (2\pi)^{2|\mathbf{a}|}  (A_jS_k^T\cdot)^{2\mathbf{a}} \hat{f}(A_jS_k^T\cdot)2^{-js} \hat{\psi}^c(\cdot)\right\|_{L^2(\R^2)}^2\\
&\quad + \sum_{(j,k)\in\Lambda^{''}}\nonumber 2^{\frac{3j}{2}}\cdot\left\|\sum_{|\mathbf{a}|\leq s}  (2\pi)^{2|\mathbf{a}|} (\tilde{A}_jS_k\cdot)^{2\mathbf{a}}\hat{f}(\tilde{A}_jS_k\cdot) 2^{-js} \hat{\psitilde}^c(\cdot)\right\|_{L^2(\R^2)}^2\\
&\qquad + \left\|\sum_{|\mathbf{a}|\leq s}  (2\pi)^{2|\mathbf{a}|}  (\cdot)^{2\mathbf{a}} \hat{f}\hat{\phi}^c \right\|_{L^2(\R^2)}^2 = \mathrm{IV}.
\end{align*}

Invoking Parseval's identity again and reversing the transformations \eqref{eq:transformII} as well as \eqref{eq:transformI} from earlier shows that
\begin{align*}
\mathrm{IV} = &\sum_{(j,k,m,\iota)\in \Lambda} |\langle \hat{f}, \sum_{|\mathbf{a}|\leq s} (2\pi)^{2|\mathbf{a}|}  (\cdot)^{2\mathbf{a}}  2^{-js}\widehat{\psi}^c_{j,k,m,\iota} \rangle_{L^2(\R^2)}  |^2 \\
= &\sum_{(j,k,m,\iota)\in\Lambda} |\sum_{|\mathbf{a}|\leq s} \langle \hat{f},  (-1)^{|\mathbf{a}|} (2\pi i)^{2|\mathbf{a}|}  (\cdot)^{2\mathbf{a}} 2^{-js}\widehat{\psi}^c_{j,k,m,\iota} \rangle_{L^2(\R^2)}  |^2.
\end{align*}
By standard results on derivatives and the Fourier transform we have that 
$$
(2\pi i \cdot)^{2\mathbf{a}} 2^{-js}\widehat{\psi}^c_{j,k,m,\iota} = \mathcal{F}(2^{-js} D^{2\mathbf{a}}  \psi^c_{j,k,m,\iota}).
$$
Thus we can invoke the Plancherel identity and partial integration to obtain 
\begin{align*}
    \mathrm{IV}= &\sum_{(j,k,m,\iota)\in\Lambda} |\sum_{|\mathbf{a}|\leq s} \langle f,  (-1)^{|\mathbf{a}|} 2^{-js}D^{2\mathbf{a}} \psi^c_{j,k,m,\iota} \rangle_{L^2(\R^2)}  |^2\\
    = &\sum_{(j,k,m,\iota)\in\Lambda} |\sum_{|\mathbf{a}|\leq s} \langle D^{\mathbf{a}} f,  2^{-js}D^{\mathbf{a}} \psi^c_{j,k,m,\iota} \rangle_{L^2(\R^2)}  |^2\\
    = &\sum_{(j,k,m,\iota)\in\Lambda} | \langle  f,  2^{-js} \psi^c_{j,k,m,\iota} \rangle_{H^s(\R^2)}  |^2.
\end{align*}

We still need to transition back from $\psi^c_{j,k,m,\iota}$ to $\psi_{j,k,m,\iota}$. We proceed by invoking a triangle inequality
\begin{align*}
    \mathrm{IV} \lesssim & \ \sum_{(j,k,m,\iota)\in\Lambda} | \langle  f,  2^{-js} \psi_{j,k,m,\iota} \rangle_{H^s(\R^2)}  |^2 \\
    &\quad + \sum_{(j,k,m,\iota)\in\Lambda} | \langle  f,  2^{-js} (\psi_{j,k,m,\iota} - \psi^c_{j,k,m,\iota}) \rangle_{H^s(\R^2)}  |^2\\
    =  & \ \mathrm{IV}_1 +   \mathrm{IV}_2,
\end{align*}
Using a similar estimate as for $\mathrm{I}_2$ in \eqref{eq:TermI2} we can estimate by invoking Lemma \ref{lemm:upperBdSobolev} that $\mathrm{IV}_2$ is negligible for $c$ small enough. This yields the result.
\end{proof}

To conclude this subsection we would like to examine how the auxiliary functions $\theta,\thetatilde$ and $\mu$ can be chosen such that they fulfill the frame property required in the proof of Theorem \ref{Thm:HsR2Frame}. Therefore assume $\psi, \psitilde,$ and $\phi$ fulfill the assumptions of Theorem \ref{Thm:HsR2Frame} with $\gamma+4> \alpha_{\mathrm{sh}}> \gamma>4+s$.
Then it follows that 
\begin{align*}
    \theta(x) &= \frac{1}{(-2\pi i)^s} D^{s}(\psi^1)(x_1) \phi^1(x_2),\\
    \thetatilde(x) &= \phi^1(x_1) \frac{1}{(-2\pi i)^s} D^{s}(\psi^1)(x_2),\\ 
    \mu &= \phi^1(x_1) \phi^1(x_2).
\end{align*}
Furthermore $D^{s}(\psi^1)$ and $\phi^1$ satisfy the assumptions of Theorem \ref{Thm:HsR2Frame} and therefore there exists a sampling parameter $\bar{c}>0$ such that for $c_1=c_2 \leq \bar c$ and $c = (c_1, c_2)$ the system $\mathcal{SH}( \phi, \psi, \psitilde, c)$ constitutes a frame for $L^2(\R^2)$. 
\section{Localization of shearlet and wavelet frames}\label{subsec:loc_sw}

We now turn to the proof of Proposition \ref{prop:CrossDecaySum}. For this, we will require the following technical lemma.

 \begin{lemma}\cite{DissPP}\label{lem:ShearFreqEnv}
Let $\psi \in L^2(\R^2)$ be such that there exists $C>0$ with
\begin{align*}
|\psihat(\xi_1,\xi_2)| \leq C \frac{\min\{1,|\xi_1|^{ \alpha }\}}{\max\{1, |\xi_1|^{ \beta} \} \max\{1, |\xi_2|^{ \beta} \}}, \quad \text{for a.e. } (\xi_1, \xi_2) \in \R^2,
\end{align*}
where $\beta/2 >\alpha >1$. Then, for $\iota =-1,1$,
\begin{align*}
\sum_{|k| \leq 2^{j/2}}|(\psi_{j,k,m,\iota})^\wedge(\xi_1,\xi_2)| \leq 2^{-3/4j} C' \frac{1}{\max\{1,|2^{-j}\xi_1|^{\beta/2}\}}\frac{1}{\max\{1,|2^{-j}\xi_2|^{\beta/2}\}},
\end{align*}
for a.e. $(\xi_1, \xi_2) \in \R^2$ and a constant $C'$.
\end{lemma}
\begin{proof}
See \cite[Lemma 3.2.1]{DissPP} for a proof.
\end{proof}

Now we are ready to prove Proposition \ref{prop:CrossDecaySum}.

\begin{proof}[Proof of Proposition \ref{prop:CrossDecaySum}]
First of all, by the definition of $\Theta_{\tau, t}^c$ one has that for all
$(j_{\mathrm{sh}},k,m_{\mathrm{sh}},\iota) \in \Lambda_0^c$ and $(j_{\mathrm{w}},m_{\mathrm{w}},\upsilon) \in \Theta_{\tau, t}^c$ such that $j_{\mathrm{w}} < 1/(2\tau) j_{\mathrm{sh}} + t$ we have that $\suppp (2^{-j_{\mathrm{w}}s}\omega_{j_{\mathrm{w}},m_{\mathrm{w}},\upsilon})^d \cap \suppp \psi_{j_{\mathrm{sh}},k,m_{\mathrm{sh}},\iota} = \emptyset$. Hence, we can assume in the sequel that $j_{\mathrm{w}} > 1/(2\tau) j_{\mathrm{sh}} + t$. 

Furthermore, the total number of wavelet translates for a fixed level $j_{\mathrm{w}}$ is of order $2^{2j_{\mathrm{w}}}$. W.l.o.g., the following computations can be done for $\upsilon=1.$ Using this observation and Plancherel's identity, we obtain 
\begin{align}
 & \ \sum_{(j_{\mathrm{w}},m_{\mathrm{w}},1)\in \Theta_{\tau, t }^c} \sum_{(j_{\mathrm{sh}},k,m_{\mathrm{sh}},\iota) \in \Lambda_0^c}  |\langle (2^{-j_{\mathrm{w}}s}\omega_{j_{\mathrm{w}},m_{\mathrm{w}},{1}})^d,  2^{-j_{\mathrm{sh}} s}\psi_{j_{\mathrm{sh}},k,m_{\mathrm{sh}},\iota} \rangle_{H^s(\Omega)}|^2 \nonumber\\
\lesssim & \ \sum_{j_{\mathrm{w}} = 0}^{\infty}\sum_{j_{\mathrm{sh}} =0}^{(2\tau)(j_{\mathrm{w}}-t)} \sum_{|k| \leq 2^{j_{\mathrm{sh}}/2} } 2^{2j_{\mathrm{w}}}\max_{m_{\mathrm{sh}}, m_{\mathrm{w}}}( \sum_{|\mathbf{a}| \leq s}|\langle (\cdot)^\mathbf{a} \widehat{(2^{-j_{\mathrm{w}} s}\omega_{j_{\mathrm{w}},m_{\mathrm{w}},{1}})^d}, (\cdot)^\mathbf{a} 2^{-j_{\mathrm{sh}} s}\widehat{\psi_{j_{\mathrm{sh}},k,m_{\mathrm{sh}},\iota}} \rangle_{L^2(\R^2)}|^2)\nonumber\\
\lesssim & \ \sum_{|\mathbf{a}| \leq s}\sum_{j_{\mathrm{w}} = 0}^{\infty}\sum_{j_{\mathrm{sh}} =0}^{(2\tau)(j_{\mathrm{w}}-t)} \sum_{|k| \leq 2^{j_{\mathrm{sh}}/2} } 2^{2j_{\mathrm{w}}}\max_{m_{\mathrm{sh}}, m_{\mathrm{w}}}(|\langle (\cdot)^\mathbf{a} \widehat{(2^{-j_{\mathrm{w}} s}\omega_{j_{\mathrm{w}},m_{\mathrm{w}},{1}})^d}, (\cdot)^\mathbf{a} 2^{-j_{\mathrm{sh}} s}\widehat{\psi_{j_{\mathrm{sh}},k,m_{\mathrm{sh}},\iota}} \rangle_{L^2(\R^2)}|^2). \label{eq:someEstimate123} 
\end{align}
Leveraging on the frequency decay of the corresponding shearlet atoms, applying Lemma \ref{lem:ShearFreqEnv}, and using (W3) of the Assumption \ref{ass:Wave} yields for any $|\mathbf{a}| \leq s$
\begin{align*}
\sum_{j_{\mathrm{sh}} = 0}^{(2\tau)(j_{\mathrm{w}}-t)} &  \sum_{{|k|}\leq 2^{j_{\mathrm{sh}}/2}} 2^{2j_{\mathrm{w}}}\max_{m_{\mathrm{w}},m_{\mathrm{sh}}} |\langle (\cdot)^\mathbf{a} \widehat{(2^{-j_{\mathrm{w}} s}\omega_{j_{\mathrm{w}},m_{\mathrm{w}},{1}})^d}, (\cdot)^\mathbf{a} 2^{-j_{\mathrm{sh}} s} \widehat{\psi_{j_{\mathrm{sh}},k,m_{\mathrm{sh}},\iota}} \rangle_{L^2(\R^2)}|^2\\
& \lesssim \sum_{j_{\mathrm{sh}} = 0}^{(2\tau)(j_{\mathrm{w}}-t)} 2^{- 3/2j_{\mathrm{sh}}}  \left(\int_{\R^2}  \frac{2^{-j_{\mathrm{w}} s}\xi^\mathbf{a} \min\{1, |2^{-j_{\mathrm{w}}} \xi_1|^{\alpha_{\mathrm{w}}}\}}{\max\{1, |2^{-j_{\mathrm{w}}}\xi_1|^{\beta_{\mathrm{w}}}\}\max\{1, |2^{-j_{\mathrm{w}}}\xi_2|^{\beta_{\mathrm{w}}}\}} \right. \\
& \quad \cdot \left. \frac{ 2^{-j_{\mathrm{sh}} s}\xi^\mathbf{a}}{\max\{1, |2^{- j_{\mathrm{sh}}}\xi_1|^{\beta_{\mathrm{sh}}/2}\}\max\{1, |2^{-j_{\mathrm{sh}}}\xi_2|^{\beta_{\mathrm{sh}}/2}\}} \, d\xi \right)^2  =: \mathrm{I}.
\end{align*}
By a simple computation we obtain that if $\beta_{\mathrm{w}} \geq s$
\begin{align*}
\frac{2^{-j_{\mathrm{w}} s}\xi^\mathbf{a} }{\max\{1, |2^{-j_{\mathrm{w}}}\xi_1|^{\beta_{\mathrm{w}}}\max\{1, |2^{-j_{\mathrm{w}}}\xi_2|^{\beta_{\mathrm{w}}-s}\} }\lesssim \frac{1}{\max\{1, |2^{-j_{\mathrm{w}}}\xi_1|^{\beta_{\mathrm{w}}-s}\} \max\{1, |2^{-j_{\mathrm{w}}}\xi_2|^{\beta_{\mathrm{w}}-s}\} }.
\end{align*}
We plug this estimate into the estimate above and obtain with $\beta_{\mathrm{w}}'=\beta_{\mathrm{w}}-s$ and $\beta_{\mathrm{sh}}'=\beta_{\mathrm{sh}}/2-s$:
\begin{align*}
\mathrm{I} & \lesssim \sum_{j_{\mathrm{sh}} = 0}^{(2\tau)(j_{\mathrm{w}}-t)} 2^{- 3/2j_{\mathrm{sh}}}  \left(\int_{\R^2}  \frac{ \min\{1, |2^{-j_{\mathrm{w}}} \xi_1|^{\alpha_{\mathrm{w}}}\}}{\max\{1, |2^{-j_{\mathrm{w}}}\xi_1|^{\beta_{\mathrm{w}}'}\}\max\{1, |2^{-j_{\mathrm{w}}}\xi_2|^{\beta_{\mathrm{w}}'}\}} \right. \\
& \quad \cdot \left. \frac{1}{\max\{1, |2^{- j_{\mathrm{sh}}}\xi_1|^{\beta_{\mathrm{sh}}'}\}\max\{1, |2^{-j_{\mathrm{sh}}}\xi_2|^{\beta_{\mathrm{sh}}'}\}} \, d\xi \right)^2 = \mathrm{II}. 
\end{align*}
We continue by applying the substitution $\xi \mapsto 2^{j_{\mathrm{sh}}} \xi$
\begin{align*}
\mathrm{II} &\lesssim \sum_{j_{\mathrm{sh}} = 0}^{(2\tau)(j_{\mathrm{w}}-t)} 2^{5/2j_{\mathrm{sh}}}  \left(\int_{\R^2}  \frac{\min\{1, | 2^{j_{\mathrm{sh}} - j_{\mathrm{w}}}\xi_1|^{\alpha_{\mathrm{w}}}\}}{\max\{1, |2^{j_{\mathrm{sh}} - j_{\mathrm{w}}} \xi_1|^{\beta_{\mathrm{w}}'}\}\max\{1, |2^{j_{\mathrm{sh}} - j_{\mathrm{w}}}\xi_2|^{\beta_{\mathrm{w}}'}\}} \right. \\
& \quad \cdot \left. \frac{1}{\max\{1, |\xi_1|^{\beta_{\mathrm{sh}}'}\}\max\{1, |\xi_2|^{\beta_{\mathrm{sh}}'}\}} \, d\xi \right)^2. \\
\lesssim &\sum_{j_{\mathrm{sh}} = 0}^{\infty} 2^{5/2j_{\mathrm{sh}}}   \left(\int_{\R^2}  \frac{\min\{1, |2^{j_{\mathrm{sh}} - j_{\mathrm{w}}} \xi_1|^{\alpha_{\mathrm{w}}}\}}{\max\{1, |\xi_1|^{\beta_{\mathrm{sh}}'}\}\max\{1, |\xi_2|^{\beta_{\mathrm{sh}}'}\}} \, d\xi \right)^2 \\
\lesssim &\sum_{j_{\mathrm{sh}} = 0}^{\infty} 2^{5/2j_{\mathrm{sh}} + 2\alpha_{\mathrm{w}}(j_{\mathrm{sh}}-j_{\mathrm{w}})}   \left(\int_{\R^2}  \frac{ |\xi_1|^{\alpha_{\mathrm{w}}} }{\max\{1, |\xi_1|^{\beta_{\mathrm{sh}}'}\}\max\{1, |\xi_2|^{\beta_{\mathrm{sh}}'}\}} \, d\xi \right)^2.
\end{align*}
Since  $\beta_{\mathrm{sh}}'-\alpha_{\mathrm{w}}>1$ we obtain that the integral above is finite and hence we conclude
\begin{align}
\mathrm{I} \lesssim \sum_{j_{\mathrm{sh}} = 0}^{(2\tau)(j_{\mathrm{w}}-t)} 2^{5/2j_{\mathrm{sh}} + 2\alpha_{\mathrm{w}}(j_{\mathrm{sh}}-j_{\mathrm{w}})}.\label{eq:theOtherThat}
\end{align}
We rewrite the last sum above as
\begin{align}
\sum_{j_{\mathrm{sh}} = 0}^{(2\tau)(j_{\mathrm{w}}-t)} 2^{5/2j_{\mathrm{sh}} + 2\alpha_{\mathrm{w}}(j_{\mathrm{sh}}-j_{\mathrm{w}})} = 2^{-2 \alpha_{\mathrm{w}} \epsilon j_{\mathrm{w}} }\sum_{j_{\mathrm{sh}} = 0}^{(2\tau)(j_{\mathrm{w}}-t)} 2^{5/2j_{\mathrm{sh}} + 2\alpha_{\mathrm{w}}(j_{\mathrm{sh}}-(1-\epsilon)j_{\mathrm{w}})}.\label{eq:that}
\end{align}
Since $j_{\mathrm{w}}>1/(2\tau)j_{\mathrm{sh}} + t$, we can now estimate
\begin{align*}
\sum_{j_{\mathrm{sh}} = 0}^{\infty} 2^{5/2j_{\mathrm{sh}} + 2\alpha_{\mathrm{w}}(j_{\mathrm{sh}}-(1-\epsilon)j_{\mathrm{w}})} \lesssim 2^{-2 \alpha_{\mathrm{w}} (1-\epsilon)t} \sum_{j_{\mathrm{sh}}=0}^\infty  2^{5/2j_{\mathrm{sh}} + 2\alpha_{\mathrm{w}}(j_{\mathrm{sh}}-(1-\epsilon) (1/(2\tau)j_{\mathrm{sh}}) }.
\end{align*}
Since $\alpha_{\mathrm{w}}((1-\eps)/\tau-2) > 5/2$ by assumption, the latter sum is finite. This leads to the estimate
\begin{align}
\sum_{j_{\mathrm{sh}} = 0}^{\infty} 2^{5/2j_{\mathrm{sh}} + 2\alpha_{\mathrm{w}}(j_{\mathrm{sh}}-(1-\epsilon)j_{\mathrm{w}})}
\lesssim 2^{-2 \alpha_{\mathrm{w}} (1-\epsilon)t}. \label{eq:ThisInCombinationWithThat}
\end{align}
Now \eqref{eq:ThisInCombinationWithThat} in combination with \eqref{eq:that} and \eqref{eq:theOtherThat} implies together with \eqref{eq:someEstimate123} that
\begin{align*}
&\sum_{(j_{\mathrm{sh}},k,m,\iota) \in \Lambda_0^c} \sum_{(j_{\mathrm{w}},m',1)\in \Theta_{\tau, t }^c}  |\langle ({2^{-j_{\mathrm{w}} s}\omega_{j_{\mathrm{w}},m',\upsilon}})^d, 2^{-j_{\mathrm{sh}} s}\psi_{j_{\mathrm{sh}},k,m,\iota} \rangle_{H^s(\Omega)}|^2 \\
&\lesssim \sum_{j_{\mathrm{w}} = 0}^\infty 2^{-2 \alpha_{\mathrm{w}} \epsilon j_{\mathrm{w}} } 2^{-2 \alpha_{\mathrm{w}} (1-\epsilon)t}\lesssim 2^{-2 \alpha_{\mathrm{w}} (1-\epsilon)t}.
\end{align*}
\end{proof}

\end{document}